\documentclass[reqno, 12pt]{amsart}
\makeatletter
\let\origsection=\section 
\def\section{\@ifstar{\origsection*}{\mysection}}
\def\mysection{\@startsection{section}{1}\z@{.7\linespacing\@plus\linespacing}{.5\linespacing}{\normalfont\scshape\centering\S}}
\makeatother

\usepackage{amsmath,amssymb,amsthm}
\usepackage{mathrsfs}
\usepackage{dsfont}
\usepackage{mathabx}\changenotsign

\usepackage{microtype}

\usepackage{xcolor}
\usepackage[backref=page]{hyperref}
\hypersetup{
    colorlinks,
    linkcolor={red!60!black},
    citecolor={green!60!black},
    urlcolor={blue!60!black}
}

\usepackage{bookmark}

\usepackage[abbrev,msc-links,backrefs]{amsrefs}
\usepackage{doi}

\renewcommand{\PrintDOI}[1]{\doi{#1}}

\usepackage[T1]{fontenc}
\usepackage{lmodern}

\usepackage[english]{babel}

\usepackage{tikz}

\usepackage{geometry}
\usepackage{enumitem}
\newcommand\rmlabel{\upshape({\itshape \roman*\,\/})}

\let\polishlcross=\l
\def\l{\ifmmode\ell\else\polishlcross\fi}

\newcommand\qand{\quad\text{and}\quad}
\newcommand\qqand{\qquad\text{and}\qquad}

\let\emptyset=\varnothing
\let\setminus=\smallsetminus
\let\backslash=\smallsetminus

\makeatletter
\def\moverlay{\mathpalette\mov@rlay}
\def\mov@rlay#1#2{\leavevmode\vtop{
   \baselineskip\z@skip \lineskiplimit-\maxdimen
   \ialign{\hfil$\m@th#1##$\hfil\cr#2\crcr}}}
\newcommand{\charfusion}[3][\mathord]{
    #1{\ifx#1\mathop\vphantom{#2}\fi
        \mathpalette\mov@rlay{#2\cr#3}
      }
    \ifx#1\mathop\expandafter\displaylimits\fi}
\makeatother

\newcommand{\dcup}{\charfusion[\mathbin]{\cup}{\cdot}}

\newtheorem{theorem}             {Theorem}
\newtheorem{lemma}     [theorem] {Lemma}

\newtheorem{claim}[theorem] {Claim}

\newtheorem{definition}[theorem]{Definition}
\newtheorem{fact}[theorem]{Fact}
\newtheorem*{observation}{Observation}

\newcommand\cca{{\mathcal{A}}}
\newcommand\ccc{{\mathcal{C}}}
\newcommand\ccd{{\mathcal{D}}}
\newcommand\ccf{{\mathcal{F}}}
\newcommand\cch{{\mathcal{H}}}
\newcommand\cck{{\mathcal{K}}}

\newcommand\ccn{{\mathcal{N}}}
\newcommand\ccp{{\mathcal{P}}}
\newcommand\ccr{{\mathcal{R}}}
\newcommand\cct{{\mathcal{T}}}
\newcommand\ccq{{\mathcal{Q}}}
\newcommand\ccv{{\mathcal{V}}}
\newcommand\ccw{{\mathcal{W}}}
\newcommand\bbn{{\mathds{N}}}

\let\phi=\varphi

\newcommand{\Xibad}{\Xi_{\textsc{bad}}}
\newcommand{\homit}{{\textrm{hom}}}

\begin{document}

\title[Loose Hamiltonian cycles forced by $(k-2)$-degree]{Loose Hamiltonian cycles forced by large $(k-2)$-degree \\ -- approximate version --}

\author[J.~de~O.~Bastos]{Josefran de Oliveira Bastos}
\author[G.~O.~Mota]{Guilherme Oliveira Mota}
\address{Instituto de Matem\'atica e Estat\'{\i}stica, Universidade de
   S\~ao Paulo, S\~ao Paulo, Brazil}
\email{\{josefran|mota\}@ime.usp.br}

\author[M.~Schacht]{Mathias Schacht}
\author[J.~Schnitzer]{Jakob Schnitzer}
\author[F.~Schulenburg]{Fabian Schulenburg}
\address{Fachbereich Mathematik, Universit\"at Hamburg, Hamburg, Germany}
\email{schacht@math.uni-hamburg.de}
\email{\{jakob.schnitzer|fabian.schulenburg\}@uni-hamburg.de}

\thanks{The first author was supported by CAPES\@.
The second author was supported by FAPESP (Proc. 2013/11431-2 and 2013/20733-2) and CNPq (Proc. 477203/2012-4 and {456792/2014-7}).
The third author was supported through the Heisenberg-Programme of the DFG\@.
The cooperation was supported by a joint CAPES/DAAD PROBRAL (Proc. 430/15).}

\begin{abstract}
We prove that for all $k\geq 4$ and $1\leq\ell<k/2$, every $k$-uniform hypergraph~$\cch$ on~$n$ vertices with $\delta_{k-2}(\cch)\geq\left(\frac{4(k-\ell)-1}{4(k-\ell)^2}+o(1)\right)\binom{n}{2}$ contains a Hamiltonian $\ell$-cycle if $k-\ell$ divides $n$. This degree condition is asymptotically best possible.
The case $k=3$ was addressed earlier by Bu\ss\ et al.
\end{abstract}


\maketitle

\section{Introduction}
A $k$-uniform hypergraph $\cch$ is a pair $(V,E)$ with vertex set $V$ and edge set $E$ such that each edge is a subset of $k$ vertices.
Given a $k$-uniform hypergraph $\cch=(V,E)$ and $S \in\binom{V}{s}$, we denote by $\deg(S)$ the number of edges of $\cch$ containing $S$ and we denote by $N(S)$ the $(k-s)$-element sets $T\in \binom{V}{k-s}$ such that $T\dcup S$ is an edge of $E$, so $\deg(S)=|N(S)|$.
We define the \emph{minimum $s$-degree} of $\cch$, denoted by $\delta_s(\cch)$, as the minimum of $\deg(S)$ over all $s$-vertex sets $S\in \binom{V}{s}$.

We say that a $k$-uniform hypergraph is an \emph{$\ell$-cycle} if there exists a cyclic ordering of its vertices such that every edge is composed of $k$ consecutive vertices, two (vertex-wise) consecutive edges share exactly $\ell$ vertices, and every vertex is contained in an edge.
If the ordering is not cyclic, we call it an \emph{$\ell$-path} and we say that the first and last $\ell$ vertices are the ends of the path.

We are interested in the problem of finding minimum degree conditions that ensure the existence of Hamiltonian cycles, i.e.\ cycles containing all vertices of the given hypergraph.
Problems of this type attracted considerable attention in the literature over the last
two decades (see, e.g.,~\cites{RRsurv,MR3526407} and the references therein).
This problem was first studied by Katona and Kierstead in~\cite{KaKi99}.
They posed a conjecture, which was confirmed by the following result of R\"odl, Ruci\'nski, and Szemer\'edi~\cites{RoRuSz06,RoRuSz08}:
For every $k\geq 3$, if $\cch$ is a $k$-uniform $n$-vertex hypergraph with $\delta_{k-1}(\cch)\geq {(1/2+o(1))}n$, then $\cch$ contains a Hamiltonian $(k-1)$-cycle.
Their proof introduces the so-called \emph{Absorbing Method}, which we will use in our proof as well.
In~\cite{KuOs06} K\"uhn and Osthus investigated a similar question for $1$-cycles, proving that $3$-uniform hypergraphs~$\cch$ with $\delta_2(\cch)\geq {(1/4 +o(1))}n$ contain a Hamiltonian $1$-cycle.
This result was generalized to arbitrary~$k$ and $\ell$-cycles with $1\leq \ell <k/2$ by H\`an and Schacht~\cite{HaSc10} (see also~\cite{KeKuMyOs11}).

\begin{theorem}
\label{theorem:Han-Schacht}
For all integers $k\geq 3$ and $1\leq \ell<k/2$ and every $\gamma>0$ there exists an $n_0$ such that every $k$-uniform hypergraph $\cch=(V,E)$ on $|V|=n\geq n_0$ vertices with $n\in(k-\ell)\bbn$ and
\begin{equation*}
\delta_{k-1}(\cch)\geq\left(\frac{1}{2(k-\ell)}+\gamma\right)n
\end{equation*}
contains a Hamiltonian $\ell$-cycle.
\qed
\end{theorem}

To see the asymptotic optimality of the minimum degree condition, we consider the following well-known example.
Let $\cch_{k,\ell}=(V,E)$ be a $k$-uniform hypergraph on $n$ vertices such that~$E$ is the set of all edges with at least one vertex from $A \subset V$, where $|A|=\left\lceil \frac{n}{2(k-\ell)}-1 \right\rceil$.
Note that an $\ell$-cycle on $n$ vertices contains $n/(k-\ell)$ edges and for $\ell<k/2$ every vertex is contained in at most two edges of any $\ell$-cycle.
So the hypergraph $\cch_{k,\ell}$ does not contain a Hamiltonian $\ell$-cycle and has $\delta_{k-1}(\cch_{k,\ell}) = \left\lceil \frac{n}{2(k-\ell)}-1 \right\rceil$.
In~\cite{HaZh15} Han and Zhao proved a version of Theorem~\ref{theorem:Han-Schacht} with this sharp degree condition.

K\"uhn, Mycroft, and Osthus~\cite{KuMyOs10} generalized Theorem~\ref{theorem:Han-Schacht} to $1\leq \ell<k-1$, solving the problem of finding minimum $(k-1)$-degree conditions that ensure the existence of Hamiltonian $\ell$-cycles in $k$-uniform hypergraphs.

A natural question is to ask for minimum $d$-degree conditions forcing the existence of Hamiltonian $\ell$-cycles for $d<k-1$.
In this direction Bu\ss, H\`an, and Schacht proved the following asymptotically optimal result in~\cite{BuHaSc13}.

\begin{theorem}
\label{theorem:Buss-Han-Schacht}
For all $\gamma>0$ there exists an $n_0$ such that every $3$-uniform hypergraph $\cch=(V,E)$ on $|V|=n\geq n_0$ vertices with $n\in 2\bbn$ and
\begin{equation*}
\delta_{1}(\cch)\geq\left(\frac{7}{16}+\gamma\right)n
\end{equation*}
contains a Hamiltonian $1$-cycle.
\qed
\end{theorem}

Note that the asymptotic optimality again follows from the hypergraph $\cch_{k,\ell}$ considered above for~$k=3$ and $\ell=1$.
The sharp bound for $\delta_{1}(\cch)$ was proved by Han and Zhao in~\cite{HaZh15b}.
We generalize Theorem~\ref{theorem:Buss-Han-Schacht} to $k$-uniform hypergraphs  and give an asymptotically optimal bound on the minimum $(k-2)$-degree for the existence of Hamiltonian $\ell$-cycles for all $1\leq\ell<k/2$.

\begin{theorem}[Main result]
\label{theorem:main}
For all integers $k\geq 4$ and $1\leq \ell<k/2$ and every $\gamma>0$ there exists an $n_0$ such that every $k$-uniform hypergraph $\cch=(V,E)$ on $|V|=n\geq n_0$ vertices with $n\in(k-\ell)\bbn$ and
\begin{equation*}
\delta_{k-2}(\cch)\geq\left(\frac{4(k-\ell)-1}{4(k-\ell)^2}+\gamma\right)\binom{n}{2}
\end{equation*}
contains a Hamiltonian $\ell$-cycle.
\end{theorem}

The hypergraph $\cch_{k,\ell}$ motivates the following notion of extremality.
Let $k\geq 3$ and $\ell\geq 1$ be integers and let $0<\xi<1$.
A $k$-uniform hypergraph $\cch=(V,E)$ is called \emph{$(\ell,\xi)$-extremal} if there exists a set $B\subset V$ such that $|B|=\big\lfloor \frac{2(k-\ell)-1}{2(k-\ell)}n\big\rfloor$ and $e(B)\leq \xi \binom{n}{k}$, where $e(B)$ stands for the number of edges in the subhypergraph of $\cch$ induced by $B$.
Our main result follows directly from the following theorem.

\begin{theorem}
\label{theorem:hamilton}
For any $0<\xi<1$ and all integers $k\geq 4$ and $1\leq \ell<k/2$, there exists $\gamma>0$ such that the following holds for sufficiently large $n$.
Suppose $\cch$ is a $k$-uniform hypergraph on $n$ vertices with $n\in(k-\ell)\bbn$ such that $\cch$ is not $(\ell,\xi)$-extremal and
\begin{equation*}
\delta_{k-2}(\cch)\geq\left(\frac{4(k-\ell)-1}{4(k-\ell)^2}-\gamma\right)\binom{n}{2}.
\end{equation*}
Then $\cch$ contains a Hamiltonian $\ell$-cycle.
\end{theorem}

We remark that for $k=3$ and $\ell=1$, the corresponding version of Theorem~\ref{theorem:hamilton} appeared in the so-called non-extremal case of the sharp version of Theorem~\ref{theorem:Buss-Han-Schacht} in~\cite{HaZh15}.
As a result, it will be sufficient to address the extremal case for a sharp version of Theorem~\ref{theorem:main} and we shall return to this in the near future~\cite{BMSCH3b}.
For details about this approach see~\cites{HaZh15,HaZh15b}.
It is easy to check that if $\delta_{k-2}(\cch)\geq\left(\frac{4(k-\ell)-1}{4{(k-\ell)}^2}+\gamma\right)\binom{n}{2}$, then there exists $\xi=\xi(k,\ell,\gamma)>0$ such that $\cch$ is not $(\ell,\xi)$-extremal.
Consequently, Theorem~\ref{theorem:main} follows from Theorem~\ref{theorem:hamilton}.

\section{Main lemmas}
\label{sec:main}

\subsection{Outline of the proof of Theorem~\ref{theorem:hamilton}}
\label{sec:pf-outline}

The proof follows the \textit{Absorbing Method} introduced by R\"odl, Ruci\'nski, and Szemer\'edi in~\cite{RoRuSz06}.
For this, we derive the following lemmas: the Absorbing Lemma~(Lemma~\ref{lemma:absorbing}), the Reservoir Lemma~(Lemma~\ref{lemma:reservoir}), and the Path-Tiling Lemma~(Lemma~\ref{lemma:tiling}).

We call an $\ell$-path $\cca\subseteq \cch$ a \emph{$\beta$-absorbing path} for a $k$-uniform hypergraph~$\cch$ if for every subset $U\subset V(\cch)$ of size at most $\beta n$ there exists an $\ell$-path~$\ccq$ such that $V(\ccq)=V(\cca)\cup U$ and~$\ccq$ has the same ends as $\cca$, for some $\beta > 0$.
The Absorbing Lemma~(Lemma~\ref{lemma:absorbing}) ensures the existence of a $\beta$-absorbing path~$\cca$.
This reduces the problem of finding a Hamiltonian~$\ell$-cycle to that of finding an almost spanning $\ell$-cycle that contains~$\cca$.

To obtain an almost spanning $\ell$-cycle, we first find a bounded number (independent of~$|V(\cch)|$) of $\ell$-paths covering almost all vertices of~$V(\cch)\setminus \cca$ and then connect them using only vertices from a small set, a so-called \emph{reservoir set} that we fix beforehand.
The Reservoir Lemma~(Lemma~\ref{lemma:reservoir}) shows that it is possible to find this reservoir set~$R$ such that any bounded number of disjoint $\ell$-paths can be connected to an $\ell$-cycle, only using vertices from~$R$.

We can choose the sizes of~$\cca$ and~$R$ small enough, so that the remaining hypergraph satisfies almost the same degree condition as~$\cch$.
Then the Path-Tiling Lemma~(Lemma~\ref{lemma:tiling}) ensures the existence of a collection of $\ell$-paths covering almost all vertices of $V(\cch)\setminus~{(\cca \cup R)}$.
This is the only point in the proof where we use the exact value of the degree condition and the non-extremality of~$\cch$.
(In fact, a proof for the corresponding version of the Path-Tiling Lemma for a direct proof of Theorem~\ref{theorem:main}, which allows us to utilise a
slightly larger degree condition, is a bit simpler.)

As mentioned before, the paths from the Path-Tiling Lemma and~$\cca$ can be connected by using vertices from~$R$ to an almost spanning $\ell$-cycle containing~$\cca$.
Since this $\ell$-cycle contains almost all vertices of $\cch$, the absorbing property of~$\cca$ allows us to absorb the leftover vertices, i.e.\ vertices that are not contained in any of the $\ell$-paths and vertices that were not used to connect the $\ell$-paths.
The resulting $\ell$-cycle is the desired Hamiltonian~$\ell$-cycle.

\subsection{Connecting}
\label{sec:con}

In order to construct an almost spanning $\ell$-cycle of a $k$-uniform hypergraph~$\cch$, we first find some $\ell$-paths and connect them at their ends.
Recall that, given an $\ell$-path $\ccp=v_1\cdots v_t$ in $\cch$, the ends of~$\ccp$ are the sets $\{v_1,\ldots,v_\ell\}$ and $\{v_{t-\ell+1},\ldots,v_t\}$.
As usual, the \emph{size} of an $\ell$-path is the number of its edges.
For a collection of $2m$ mutually disjoint sets of $\ell$ vertices $X_i,Y_i$ we say that a set of $\ell$-paths $\cct_1,\ldots,\cct_m$ \emph{connects} ${(X_i,Y_i)}_{i\in[m]}$ if all paths are vertex-disjoint and $X_i$ and $Y_i$ are the ends of~$\cct_i$, for all $i\in[m]$.
The connections for a given collection of disjoint $\ell$-paths are given by the following lemma.
In addition the lemma allows to restrict the edges used for the connection to a given ``well-connected'' subset~$R$ of vertices.

\begin{lemma}[Connecting Lemma]
\label{lemma:connecting}
Let $\eta>0$ and let $k\geq 4$, $1\leq\ell<k/2$, and $m\geq 1$ be integers.
Let $\cch=(V,E)$ be a $k$-uniform hypergraph and $R\subset V$ with $\vert R\vert=r\geq 32km/\eta^3$.
For every collection of~$2m$ mutually disjoint sets $X_i$, $Y_i\in \binom{V}{\ell}$ the following holds for~${V'=\bigcup_{i\in[m]} (X_i\cup Y_i)\cup R}$.

If $\big|N(K)\cap \binom{R}{2}\big|\geq \eta\binom{r}{2}$ for all $K\in\binom{V'}{k-2}$, then there exist $\ell$-paths $\cct_1,\ldots,\cct_m$ of size at most four connecting ${(X_i,Y_i)}_{i\in[m]}$, which contain vertices from $V'$ only.
\end{lemma}

\begin{proof}
Given $\eta>0$ and integers $k\geq 4$, $1\leq \ell<k/2$ and $m\geq 1$, let $\cch=(V,E)$, $R\subseteq V$, and $X_i,Y_i$ for $i\in[m]$ satisfy the assumptions of the lemma.

Suppose we have constructed $\ell$-paths $\cct_1,\ldots,\cct_{j-1}$ each of size at most four connecting the pairs ${(X_i,Y_i)}_{i\in[j-1]}$ for some $j\leq m$ using only vertices from $\bigcup_{i\in[m]} (X_i\cup Y_i)\cup R$.
We want to construct a path $\cct_j$ with ends~$X_j$ and~$Y_j$.
We define $F_j=\bigcup_{i\in[m]} (X_i\cup Y_i)\cup \bigcup_{i\in [j-1]}V(\cct_i)$ as the set of \emph{forbidden vertices for~$\cct_j$}.

If $k-2\geq 2\ell= |X_j\cup Y_j|$, fix a set~$Z$ of size $k-2-2\ell$ from~$R\setminus F_j$.
Since $|R|=r\geq  32km/\eta^3$, we know that
\[
    \left|N(X_j\cup Y_j\cup Z)\cap \binom{R}{2}\right|\geq \eta\binom{r}{2}>\binom{r}{2}-\binom{r-4km }{2}\geq \binom{r}{2}-\binom{|R\setminus F_j|}{2}.
\]
Hence, there exists a hyperedge $X_j\cup Y_j\cup Z'$ with $Z'\subseteq R\setminus F_j$, which realizes the path~$\cct_j$.

It is left to consider the case that $2\ell=k-1$.
See Figure~\ref{fig:Connecting} for a drawing of the path we will construct in this case.
For a set $A\subseteq V$, let $N_A(S) = N(S) \cap \binom{A}{k-|S|}$.

\begin{figure}
    \begin{tikzpicture}
        \draw (0,0) ellipse (30pt and 15 pt) node (M') [] {$M'$};
        \draw (0,1.75) ellipse (25pt and 12 pt) node (L) [] {$L$};
        \draw (0,-1.75) ellipse (25pt and 12 pt) node (L') [] {$L'$};
        \draw (-3,1.75) ellipse (30pt and 15 pt) node (X) [] {$X_j$};
        \draw (3,-1.75) ellipse (30pt and 15 pt) node (Y) [] {$Y_j$};

        \draw [fill=black](-1.65,1.75) circle(2pt) node[below=2pt]{$v$};
        \draw [fill=black](-1.1,1.75) circle(2pt) node[below=2pt]{$x$};
        \draw [fill=black](1.25,1.75) circle(2pt) node[below=2pt]{$y$};;

        \draw [fill=black](-1.25,-1.75) circle(2pt) node[label=$x'$]{};;
        \draw [fill=black](1.1,-1.75) circle(2pt) node[label=$y'$]{};;
        \draw [fill=black](1.65,-1.75) circle(2pt) node[label=$v'$]{};;

        \pgfsetcornersarced{\pgfpoint{5mm}{5mm}}
        \draw (-4.25,1) rectangle (1,2.5);
        \draw (4.25,-1) rectangle (-1,-2.5);
        \draw (-1.35,2.275) rectangle (2,-0.7);
        \draw (1.35,-2.275) rectangle (-2,0.7);
    \end{tikzpicture}
    \caption{The path connecting $(X_j,Y_j)$.}
    \label{fig:Connecting}
\end{figure}
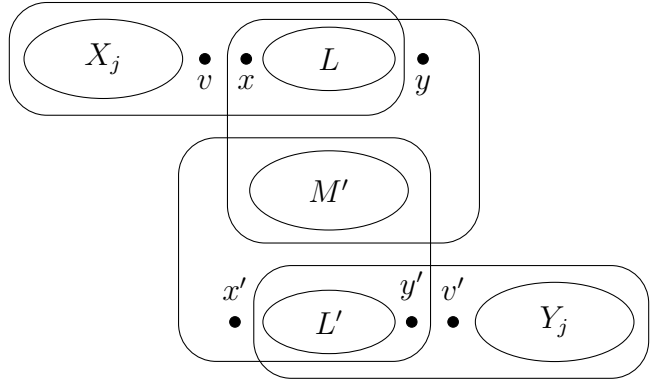

\begin{observation}
\label{obs:proof_connecting}
For any $Z\in\{X_j,Y_j\}$ and $L\in\binom{R\setminus F_j}{\ell-1}$, there are at least~$\eta r/4$ many vertices $z\in R\setminus(F_j \cup L)$ with $|N_{R\setminus F_j}(Z\cup L \cup \{z\})|\geq \eta r /4$.
\end{observation}

To see the observation note that we can consider~$N_{R\setminus F_j}(Z\cup L)$ as the edge set of a $2$-graph with vertex set $R\setminus (F_j\cup L)$.
Since $r\geq 32km/\eta^3$, it follows from the degree condition of~$\cch$ into the set~$R$ that this graph has edge density at least~$\eta/2$
and the observation follows.

Let $L\in \binom{R\setminus F_j}{\ell-1}$ and let $x$, $y\in R\setminus (F_j \cup L)$ be distinct.
We say that $(x,L,y)$ is an \emph{extendable triple} in $R\setminus F_j$ if
\[
    |N_{R\setminus F_j}(X_j\cup L\cup \{x\})|\geq \eta r/4
    \qand
    |N_{R\setminus F_j}(Y_j\cup L\cup \{y\})|\geq \eta r/4.
\]
The observation yields at least $(\eta r/4)(\eta r/4-1)>(\eta r/8)^2$
extendable triples $(x,L,y)$ for any fixed $L\in \binom{R\setminus F_j}{\ell-1}$.

Given $S \in \binom{R\setminus F_j}{\ell -2}$ and an extendable triple $(x,L,y)$ disjoint from $S$, $S\cup L\cup\{x,y\}$ is a $(k-2)$-element set.
Consequently, the minimum degree condition of the lemma yields at least $\eta\binom{r}{2}$ pairs $M\in\binom{R}{2}$ such that $S\cup M\cup L\cup \{x, y\}$ is an edge of $\cch$.
Moreover, similarly as in the proof of the observation at least $(\eta/2)\binom{|R\setminus F_j|}{2}$ of these pairs avoid $F_j$.
Since this is true for every extendable triple and there are at least $\binom{|R\setminus F_j|}{\ell-1}(\eta r/8)^2$ extendable triples, there exists an $M\in \binom{R\setminus F_j}{2}$ that, together with $S$, forms an edge of $\cch$ with at least $(\eta/2)(\eta r/8)^2 \binom{|R\setminus F_j|}{\ell-1}$ extendable triples.
Since $r\geq 32km/\eta^3$, this is more than the number of triples that any single extendable triple can intersect with, so there exist two completely disjoint extendable triples $(x,L,y)$ and $(x',L',y')$ that form an edge of $\cch$ together with~$M' = M \cup S$.

By the definition of extendable triples we have
\begin{align*}
    \big|N_{R\setminus F_j}(X_j\cup L\cup\{x\})\big|&\geq \eta r/4>k+1=\big|M'\cup L'\cup\{x',y',y\}\big|\\
\intertext{and}
    \big|N_{R\setminus F_j}(Y_j\cup L'\cup\{y'\})\big|&\geq \eta r/4>k+2=\big|M'\cup L\cup\{x,y,x'\}\big|+1.
\end{align*}
Consequently there are $v,v'\in R\setminus F_j$ such that the hyperedges
\[
    \{X_j\cup L\cup\{v,x\}\},\quad
    \{M'\cup L\cup\{x,y\}\},\quad
    \{M'\cup L'\cup\{x',y'\}\},\qand
    \{Y_j\cup L'\cup\{y',v'\}\}
\]
are edges of $\cch$, which form a path of size $4$ connecting
$(X_j,Y_j)$.
\end{proof}

In the main proof we will connect $\ell$-paths to an almost spanning $\ell$-cycle.
The Reservoir Lemma (stated below) ensures the existence of a small set~$R$ such that we can connect an arbitrary collection of at most $2m$ many $\ell$-sets, only using vertices of~$R$.

\begin{lemma}[Reservoir Lemma]
\label{lemma:reservoir}
Let $\eta,\varepsilon>0$ and let $k\geq 4$, $1\leq\ell<k/2$, and $m\geq 1$ be integers.
Then for every sufficiently large $k$-uniform hypergraph $\cch=(V,E)$ on $n$ vertices with $\delta_{k-2}(\cch)\geq \eta \binom{n}{2}$ there is a set $R\subset V$ with $|R|\leq \varepsilon n$ such that the following holds.

For every collection $X_i,Y_i$ for ${i\in[j]}$ of $2j$ mutually disjoint sets of $\ell$ vertices, where $j\leq m$, there exist $\ell$-paths $\cct_1,\ldots,\cct_j$ of size at most $4$ connecting $(X_i,Y_i)_{i\in[j]}$ that, moreover, contain vertices from $\bigcup_{i\in[j]} (X_i\cup Y_i)\cup R$ only.
\end{lemma}

Lemma~\ref{lemma:reservoir} is a consequence of Lemma~\ref{lemma:connecting}, since one can show that with high probability a suitably sized random subset~$R\subseteq V$ inherits an appropriately scaled minimum degree condition from $\cch$.
As a consequence such a set satisfies the assumptions of Lemma~\ref{lemma:connecting} (with $\eta/2$) and the lemma yields the conclusion of Lemma~\ref{lemma:reservoir} (see, e.g.~\cite{BuHaSc13}*{Lemma~6} for a very similar argument).

\subsection{Absorption}
\label{sec:abs}

Given a $k$-uniform hypergraph~$\cch$ and $U \subset V$ with $|U|\in (k-\ell)\mathds{N}$, we say that an $\ell$-path~$\cca$ \emph{absorbs}~$U$ if there exists an $\ell$-path~$\ccq$ with the same ends as~$\cca$ and $V(\ccq) = V(\cca) \cup U$.
At the end of the main proof we will absorb all vertices outside of an almost spanning $\ell$-cycle to obtain a Hamiltonian $\ell$-cycle using an \emph{absorbing path}~$\cca$, i.e.\ a path that can absorb any set~$U$ of small linear size.
The existence of such a path~$\cca$ is given by the following lemma.

\begin{lemma}[Absorbing Lemma]
\label{lemma:absorbing}
For every $\eta$, $\zeta>0$ and all integers $k\geq 4$ and $1\leq\ell<k/2$ there exists $\varepsilon>0$ such that the following holds for sufficiently large $n$.
Let $\cch=(V,E)$ be a $k$-uniform hypergraph on $n$ vertices that satisfies $\delta_{k-2}(\cch)\geq \eta \binom{n}{2}$.
Then there is an $\ell$-path~$\mathcal{\cca}$ with $|V(\cca)|\leq \zeta n$ such that for all subsets $U\subset V\setminus V(\cca)$ of size at most $\varepsilon n$ with $|U|\in (k-\ell)\bbn$ there exists an $\ell$-path $\ccq \subset\cch$ with $V(\ccq)=V(\cca)\cup U$ such that $\cca$ and $\ccq$ have the same ends.
\end{lemma}

\begin{proof}
Let $\eta, \zeta > 0$ and let $k \geq 4$ and $1 \leq \ell < k/2$ be integers, and assume w.l.o.g.\ that $\eta, \zeta \le 1$.
Fix auxiliary constants
\[
    \widetilde{\eta}=\frac{\eta}{4k!} \qand q = 3k - 2\ell
\]
and set
\[
    \varepsilon = \frac{\zeta \widetilde{\eta}^{10} }{56k q^2}.
\]
Let $n$ be sufficiently large and let $\cch = (V, E)$ be a $k$-uniform hypergraph on $n$ vertices that satisfies $\delta_{k - 2}(\cch) \geq \eta \binom{n}{2}$.
First, we will show that for any $S \in \binom{V}{k - \ell}$ there exist many, i.e.~$\Omega(n^q)$, $3$-edge $\ell$-paths that absorb $S$
(see Claim~\ref{claim:absorber} below).
For that we will use the following consequence of the minimum degree condition.
Let $A, B \subset V(\cch)$ be disjoint sets of vertices with $|A| \leq k - 2$ and $|B| \leq q+k$.
Then,
\begin{equation}
    \label{eq:consequence_mindegree}
    \deg_{\cch[V\backslash B]}(A) \geq \frac{(n-|A|)\cdot\,\cdots\,\cdot (n-k+3)}{(k-|A|)!} \cdot \eta \binom{n}{2}-|B|n^{k-|A|-1}  \geq \widetilde{\eta} n^{k - |A|}.
\end{equation}

\begin{claim}
\label{claim:absorber}
For every $S \in \binom{V}{k - \ell}$ there exist at least $\widetilde{\eta}^5 n^{q}$ many $3$-edge $\ell$-paths that absorb~$S$.
\end{claim}
\begin{proof}
Let $S_1\cup S_2 = S$ be chosen in some way such that
\begin{equation}
    \label{eq:s1s2}
    |S_1| \geq |S_2| \geq |S_1| - 1 \qand \max\{0, 3\ell - k \} \leq |S_1 \cap S_2| < \ell
\end{equation}
and set $s_1 = |S_1|$, $s_2 = |S_2|$, and $s_3 = |S_1 \cap S_2|$.
Clearly, we have
\begin{equation}
    \label{eq:size_S}
    s_1+s_2-s_3=|S|=k-\ell.
\end{equation}
It follows from the choices above that $s_1+s_2\geq 2\ell$. Indeed, since $s_3\geq 3\ell-k$
we have $k-\ell=s_1+s_2-s_3\leq s_1+s_2-3\ell+k$ and, hence, $s_1+s_2\geq 2\ell$.
Furthermore, $s_1\geq s_2\geq s_1-1$ yields
\begin{equation}
    \label{eq:s1s2ell}
    s_1\geq s_2\geq \ell.
\end{equation}

Consequently, $|S_1| > |S_1 \cap S_2|$ (see \eqref{eq:s1s2}) and $s_1 < k-\ell$ by \eqref{eq:size_S}.
We then select the following sets.
See Figure~\ref{fig:Absorbing} for a drawing of the chosen sets and edges containing them.
In each step, we will only select sets that are disjoint from~$S$ and everything chosen in previous steps.
\begin{enumerate}[label=\rmlabel]
    \item\label{it:abs1} Since $s_1 \leq k - \ell - 1 \leq k - 2$, by~\eqref{eq:consequence_mindegree} there exist $\widetilde{\eta} n^{k - s_1}$ choices for a $(k-s_1)$-set $X$ such that $f_1 = X \dcup S_1$ is an edge of $\cch$.
        Since $|X|=k-s_1\overset{\eqref{eq:size_S}}{=}\ell+s_2-s_3$ it follows from~\eqref{eq:s1s2ell} that we may partition $X= L_1 \dcup F \dcup F_1$ such that $|L_1| = \ell$, $|F| = \ell - s_3 \overset{\eqref{eq:s1s2}}{>} 0$, and $|F_1| = s_2 - \ell \ge 0$.
    \item Since $k \geq 4$ we have $k-\ell\geq 3$ and, consequently, $s_1\geq \lceil (k-\ell)/2\rceil \geq 2$.
        Thus, by~\eqref{eq:consequence_mindegree} and $|S_2\cup F|=s_2 + \ell - s_3 = k - s_1$, there exist $\widetilde{\eta} n^{s_1}$ choices for a set~$Y$ of size~$s_1$ such that $f_2 = S_2 \dcup F \dcup Y$ is an edge of $\cch$.
        Again owing to~\eqref{eq:s1s2ell} we may partition $Y=L_2 \dcup F_2$ such that $|L_2| = \ell$ and $|F_2| = s_1 - \ell \ge 0$.
  \item Fix $L_1' \subset L_1$ and $L_2' \subset L_2$ subsets of size $\ell - 1$.
      Note that
      \[
          |L_1' \dcup L_2' \dcup F \dcup F_1 \dcup F_2| = |X|+|Y|-2=k - 2.
      \]
      Therefore, there exist at least $\widetilde{\eta} n^2$ choices for a pair of vertices $\{x_1 ,x_2\}$ such that $e_2 = \{x_1, x_2\} \dcup L_1' \dcup L_2' \dcup F \dcup F_1 \dcup F_2$ is an edge of $\cch$.
      \item\label{it:abs4} Since $k \geq 4$ we have $\ell + 1 \leq k - 2$.
      Therefore, there exist $\widetilde{\eta} n^{k - (\ell + 1)}$ choices each for two disjoint edges $e_1$ and $e_3$ such that $\{x_1\} \dcup L_1 \subset e_1$ and $\{x_2\} \dcup L_2 \subset e_3$.
\end{enumerate}
By construction we have
\[
e_1\cap e_2=\{x_1\}\dcup L_1' \qand e_2\cap e_3=\{x_2\}\dcup L_2',
\]
so the edges $e_1,e_2,$ and $e_3$ form an $\ell$-path~$\ccp$ in~$\cch$.
Moreover, since
\[
e_1\cap f_1=L_1,\quad |f_1\cap f_2|=|(S_1\cap S_2)\cup F|\overset{\text{\ref{it:abs1}}}{=}\ell, \qand f_2\cap e_3=L_2,
\]
the edges $e_1,f_1,f_2,$ and $e_3$ form an $\ell$-path~$\ccp'$.
Since $k-\ell-1\geq \ell$, we may select for~$\ccp$ and~$\ccp'$ the same ends in~$e_1$ and~$e_3$.
Moreover, $V(\ccp')=V(\ccp)\cup S$ and, therefore, the $\ell$-path~$\ccp$ absorbs~$S$.
From~\ref{it:abs1}--\ref{it:abs4} it is clear that there are at least $\widetilde{\eta}^5 n^q$ choices for~$\ccp$.
\end{proof}

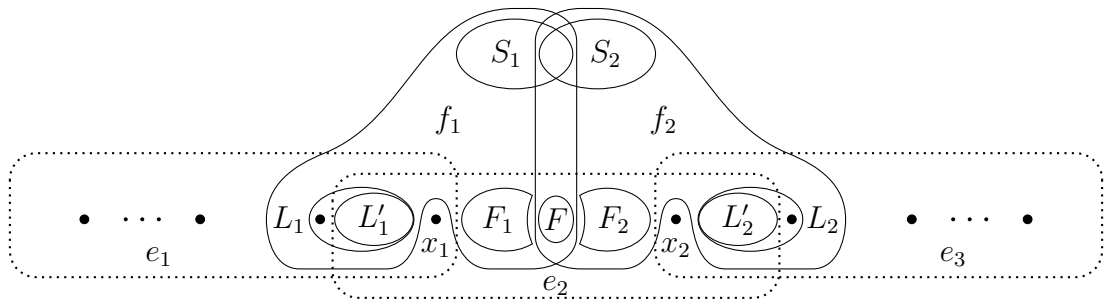
\begin{figure}
  \centering
  \begin{tikzpicture}[scale=1.1]
    \draw (-0.5,2) ellipse (20pt and 12 pt); 
    \node (S_1) at (-0.6,2) [] {$S_1$};
    \draw (0.5,2) ellipse (20pt and 12 pt); 
    \node (S_2) at (0.6,2) [] {$S_2$};

    \node (F_1) at (-0.7,0) [] {$F_1$};
    \begin{scope}
        \clip (-0.28,-0.3) arc [x radius=15pt, y radius=11pt,start angle=310, end angle=50] -- cycle;
        \draw[line width=0.8pt] (-0.28,-0.3) arc [x radius=15pt, y radius=11pt,start angle=310, end angle=50];
        \draw[line width=0.4pt] (-0.28,0.3) arc[x radius=20.5pt, y radius=20.5pt, start angle=155, end angle=206];
    \end{scope}
    \begin{scope}[xscale=-1]
        \clip (-0.28,-0.3) arc [x radius=15pt, y radius=11pt,start angle=310, end angle=50] -- cycle;
        \draw[line width=0.8pt] (-0.28,-0.3) arc [x radius=15pt, y radius=11pt,start angle=310, end angle=50];
        \draw[line width=0.4pt] (-0.28,0.3) arc[x radius=20.5pt, y radius=20.5pt, start angle=155, end angle=206];
    \end{scope}

    \node (F_2) at (0.7,0) [] {$F_2$};
    \draw[rounded corners=5pt] (0,0) ellipse [x radius=6pt, y radius=8pt] ;
    \node (F_2) at (0,0) [] {$F$};

    \draw [fill=black](-1.45,0) circle(1.5pt) node[below=3pt] {$x_1$};
    \draw [fill=black](1.45,0) circle(1.5pt) node[below=3pt] {$x_2$};

    \draw (-2.2,0) ellipse (13.5pt and 9 pt) node (L_1') [] {$L_1'$};
    \draw (-2.35,0) ellipse (18pt and 11 pt) node (L_1) [left=17pt] {$L_1$};    
    \draw [fill=black](-2.85,0) circle(1.5pt) node[label]{};

    \draw (2.2,0) ellipse (13.5pt and 9 pt) node (L_2') [] {$L_2'$};
    \draw (2.35,0) ellipse (18pt and 11 pt) node (L_2) [right=17pt] {$L_2$};    
    \draw [fill=black](2.85,0) circle(1.5pt) node[label]{};

    \draw [fill=black](-4.3,0) circle(1.5pt) node[label]{};
    \draw [fill=black](-4.8,0) circle(0.5pt) node[label]{};
    \draw [fill=black](-5,0) circle(0.5pt) node[label]{};
    \draw [fill=black](-5.2,0) circle(0.5pt) node[label]{};
    \draw [fill=black](-5.7,0) circle(1.5pt) node[label]{};

    \draw [fill=black](4.3,0) circle(1.5pt) node[label]{};
    \draw [fill=black](4.8,0) circle(0.5pt) node[label]{};
    \draw [fill=black](5,0) circle(0.5pt) node[label]{};
    \draw [fill=black](5.2,0) circle(0.5pt) node[label]{};
    \draw [fill=black](5.7,0) circle(1.5pt) node[label]{};

    \draw[rounded corners=10pt,dotted,style=thick] (-6.6,0.8) rectangle (-1.2,-0.7);
        \node (L') at (-4.8,-0.45) [] {$e_1$};
    \draw[rounded corners=10pt,dotted,style=thick] (-2.7,0.55) rectangle (2.7,-0.95);
        \node (L') at (0,-0.8) [] {$e_2$};
    \draw[rounded corners=10pt,dotted,style=thick] (6.6,0.8) rectangle (1.2,-0.7);
        \node (L') at (4.8,-0.45) [] {$e_3$};

        \draw (0.25,-0.2)
            to[in=0,out=270] (-0.3,-0.6) -- (-0.8,-0.6)
            to[out=180,in=280] (-1.25,-0.1)
            to[out=100,in=0] (-1.45,0.25)
            to[out=180,in=80] (-1.65,-0.1)
            to[out=260,in=0] (-2.05,-0.6) -- (-3.1,-0.6)
            to[out=180,in=270] (-3.5,0)
            to[out=90,in=202] (-2.8,0.8)
            to[out=22,in=180] (-0.5,2.55) -- (-0.3,2.55)
            to[out=0,in=90] (0.25,2.15)
            -- cycle;
        \begin{scope}[xscale=-1]
        \draw (0.25,-0.2)
            to[in=0,out=270] (-0.3,-0.6) -- (-0.8,-0.6)
            to[out=180,in=280] (-1.25,-0.1)
            to[out=100,in=0] (-1.45,0.25)
            to[out=180,in=80] (-1.65,-0.1)
            to[out=260,in=0] (-2.05,-0.6) -- (-3.1,-0.6)
            to[out=180,in=270] (-3.5,0)
            to[out=90,in=202] (-2.8,0.8)
            to[out=22,in=180] (-0.5,2.55) -- (-0.3,2.55)
            to[out=0,in=90] (0.25,2.15)
            -- cycle;
        \end{scope}
     \node (f_1) at (-1.3,1.2) [] {$f_1$};
     \node (f_1) at (1.3,1.2) [] {$f_2$};
  \end{tikzpicture}
  \caption{The path $\ccp$, consisting of $e_1,e_2$, and $e_3$, that absorbs $S$.}
  \label{fig:Absorbing}
\end{figure}

Following the scheme from~\cite{RoRuSz06}, let $\ccf\subset {V(\cch)}^q$ be a family of ordered $q$-sets of vertices such that each of these sets are selected from ${V(\cch)}^q$ independently with probability
\[
    p = \frac{4 \varepsilon}{\widetilde{\eta}^5n^{q - 1}}.
\]
An $\ell$-path in ${V(\cch)}^q$ is an ordered set $(v_1,\dots,v_q)$ of vertices such that
\[
    e_1=\{v_1,\dots,v_k\},\quad e_2=\{v_{k-\ell+1},\dots,v_{2k-\ell}\}, \qand e_3=\{v_{2k-2\ell+1},\dots,v_{3k-2\ell}\}
\]
are edges in~$\cch$.
Using Chernoff's inequality, with high probability we have
\[
    |\ccf| \leq 2pn^q = \frac{8\varepsilon}{\widetilde{\eta}^5}n.
\]
By Claim~\ref{claim:absorber}, for each set~$S$ of size $k-\ell$, at least $\widetilde{\eta}^5n^q$ $\ell$-paths in ${V(\cch)}^q$ absorb~$S$.
By Chernoff's inequality, w.h.p.\ for all $S \in \binom{V}{k - \ell}$, there are at least $2\varepsilon n$ $\ell$-paths in~$\ccf$ that absorb~$S$.
The expected value of the number of intersecting pairs of $q$-sets in~$\ccf$ is at most
\[
    q^2n n^{2q-2} p^2 = q^2n^{2q-1} {\left(\frac{4\varepsilon}{\widetilde{\eta}^5n^{q-1}}\right)}^2 = \varepsilon n \frac{16 \zeta}{56k} \le \frac12 \varepsilon n.
\]
So by Markov's inequality, the number of intersecting pairs of $q$-sets in~$\ccf$ is at most $\varepsilon n$ with probability at least $1/2$.

Let $\ccf$ be a family that satisfies the above conditions.
For each of the intersecting pairs in~$\ccf$, delete one of the $q$-sets and let $\ccf'\subset\ccf$ be the remaining family.
We want to use Lemma~\ref{lemma:connecting} with $R=V$, which is sufficiently large as the following calculation shows:
\[
    |\ccf'| \le \frac{8\varepsilon}{\widetilde{\eta}^5}n = \frac{8\zeta\widetilde{\eta}^5}{q^2 56k}n = \frac{8\zeta\eta^5}{q^2 56k {(4k!)}^5}n \le \frac{\eta^3}{32k} n = \frac{\eta^3}{32k} |R|.
\]
So we can connect all $\ell$-paths in $\ccf'$ to an $\ell$-path~$\cca$ with
\[
    |V(\cca)| \leq |\ccf'|\cdot (4k+q)\leq 2pn^q\cdot 7k =\frac{56k}{\widetilde{\eta}^5}\varepsilon n \leq \zeta n
\]
and this path absorbs all sets $U \subset V\backslash V(\cca)$ with $|U| \in (k - \ell)\bbn$ and $|U| \leq \varepsilon n$.
\end{proof}

\subsection{Path-Tiling}
\label{sec:til}

In this part we will find a path-tiling of $\ell$-paths in~$\cch$ that covers all but a small fraction of the vertices of $\cch$.
For that purpose we use the so-called \emph{weak regularity lemma} for hypergraphs, which is the straightforward extension of Szemer\'edi's regularity lemma for graphs~\cite{Sz75}.
Roughly speaking, we will show that there exists a fractional $\ccc_\ell$-tiling, a so-called $\beta$-$\hom(\ccc_\ell)$-tiling in the resulting reduced
hypergraph~$\ccr$ of~$\cch$, where~$\ccc_\ell$ is the $k$-uniform ``cherry''
consisting of two hyperedges that share exactly~$2\ell$ vertices.
The fractional $\ccc_\ell$-tiling of~$\ccr$ will transfer to a path-tiling of~$\cch$.

First, we introduce the standard notation for the regularity lemma.
Let $\cch=(V,E)$ be a $k$-uniform hypergraph and let $V_1,\dots ,V_k$ be non-empty, mutually disjoint subsets of~$V$.
We denote the number of edges with one vertex in each $V_i$ by $e_\cch(V_1,\dots ,V_k)$ and define the density of~$\cch$ w.r.t.\ $(V_1,\dots ,V_k)$ by
\[
d_\cch(V_1,\dots ,V_k)=\frac{e_\cch(V_1,\dots ,V_k)}{|V_1|\cdots |V_k|}.
\]

For $\varepsilon>0$ and $d>0$, a $k$-tuple $(V_1,\dots ,V_k)$ of mutually disjoint subsets of vertices is called \textit{$(\varepsilon,d)$-regular} if for all $k$-tuples $(A_1,\dots ,A_k)$ of subsets $A_i\subseteq V_i$ with $|A_i|\geq \varepsilon |V_i|$, we have
\[
|d_\cch(A_1,\dots ,A_k)-d|\leq \varepsilon.
\]
Moreover, the tuple $(V_1,\dots ,V_k)$ is called \textit{$\varepsilon$-regular} if it is $(\varepsilon,d)$-regular for some~$d>0$.
Below we state the weak hypergraph regularity lemma (see, e.g.~\cites{Ch91,FR92,St90}).
\begin{lemma}[Weak regularity lemma]
\label{lemma:regularity_lemma}
For all integers~$k\geq 2$ and~$t_0\geq 1$ and for every~$\varepsilon>0$, there exists $T_0=T_0(k,t_0,\varepsilon)$ such that for every sufficiently large $k$-uniform hypergraph $\cch =(V,E)$ on $n$ vertices, there exists a partition $V=V_0 \dcup V_1 \dcup \dots
\dcup V_t$ satisfying
\begin{enumerate}[label=\rmlabel]
    \item\label{it:weakreg1} $t_0\leq t\leq T_0$,
    \item\label{it:weakreg2} $|V_1|=\dots =|V_t|$ and $|V_0|\leq \varepsilon n$, and
    \item\label{it:weakreg3} for all but at most $\varepsilon\binom{t}{k}$ many $k$-subsets $\{i_1,\dots ,i_k\}\subset [t]$, the $k$-tuple $(V_{i_1},\dots ,V_{i_k})$ is $\varepsilon$-regular.
\end{enumerate}
\end{lemma}

A  vertex partition of a hypergraph $\cch$ satisfying~\ref{it:weakreg1}--\ref{it:weakreg3} of
the conclusion of Lemma~\ref{lemma:regularity_lemma} will be referred to as an $\varepsilon$-regular partition.
For $\varepsilon > 0$ and $d>0$, we define the \textit{reduced hypergraph} $\ccr=\ccr(\varepsilon,d)$ of~$\cch$ w.r.t.\ such a partition
as the $k$-uniform hypergraph on the vertex set $[t]$ and
\[
\{i_1,\dots ,i_k\}\in E(\ccr) \Longleftrightarrow (V_{i_1},\dots ,V_{i_k}) \text{ is } (\varepsilon,d')\text{-regular, for some }d' \geq d.
 \]

In typical applications of the regularity lemma, the reduced hypergraph inherits some key features of the given hypergraph $\cch$.
In fact, the following observation shows that the reduced hypergraph inherits approximately the minimum degree condition of the original hypergraph.
A similar result can be found in~\cite{HaSc10}*{Proposition 16} and for completeness we include its proof below.

\begin{lemma}
\label{lemma:degree_clustergraph}
Given $c,\varepsilon,d>0$ and integers $k\geq 3$ and $t_0\geq 2k/d$,
let $\cch$ be a $k$-uniform hypergraph on $n\geq t\geq t_0$ vertices such that
\begin{equation*}
    \delta_{k-2}(\cch)\geq c\binom{n}{2}.
\end{equation*}
If~$\cch$ has an $\varepsilon$-regular partition $V_0\dcup V_1\dcup\dots\dcup V_t$ with reduced hypergraph $\ccr=\ccr(\varepsilon,d)$, then at most $\sqrt{\varepsilon}\binom{t}{{k-2}}$ many $(k-2)$-subsets $K$ of $[t]$ violate
\begin{equation*}
    \deg_{\ccr}(K)\geq (c-2d - \sqrt{\varepsilon})\binom{t}{2}.
\end{equation*}
\end{lemma}

\begin{proof}
Let $\ccd =\ccd (d)$ and $\ccn =\ccn (\varepsilon)$ be the hypergraphs with vertex set~$[t]$ and
\begin{itemize}
\item $E(\ccd)$ consists of all sets $\{i_1,\dots,i_k\}$ such that $d(V_{i_1},\dots,V_{i_k})\geq d$,
\item $E(\ccn)$ consists of all sets $\{i_1,\dots,i_k\}$ such that $(V_{i_1},\dots,V_{i_k})$ is not $\varepsilon$-regular.
\end{itemize}
Note that the reduced hypergraph $\ccr(\varepsilon,d)$ is the hypergraph with vertex set~$[t]$ and edge set $E(\ccd)\setminus E(\ccn)$.
For an arbitrary $K=\{i_1,\dots ,i_{k-2}\}\in \binom{[t]}{k-2}$ we will show that
\begin{equation}
    \label{eq:degree_cluster_D}
    \deg_{\ccd}(K)\geq \left(c-2d\right)\binom{t}{2}.
\end{equation}
Let $n/t\geq |V_{i_j}| = m \geq (1-\varepsilon)n/t$ be the size of the partition classes and let~$x$ be the number of edges in~$\cch$ that intersect each $V_{i_j}$ in exactly one vertex for each $j\in [k-2]$.
By the condition on $\delta_{k-2}(\cch)$ and $t\geq t_0\geq 2k/d$, we obtain
\begin{equation*}
    x\geq m^{k-2} \left(c\binom{n}{2}-(k-2)mn \right) \geq \left(c-d\right)m^{k-2}\binom{n}{2}.
\end{equation*}
If~\eqref{eq:degree_cluster_D} did not hold, then we would find for~$x$ the upper bound
\begin{equation*}
    x< \left(c- 2d\right)\binom{t}{2} m^k+ \binom{t}{2} d m^k \leq \left(c-d\right) m^{k-2} \binom{n}{2}
\end{equation*}
contradicting the lower bound for~$x$.

Next we observe that at most $\sqrt{\varepsilon}\binom{t}{{k-2}}$ many $(k-2)$-sets~$K$ satisfy $\deg_{\ccn}(K)\leq \sqrt{\varepsilon}\binom{t}{2}$ since the number of non-$\varepsilon$-regular $k$-tuples in $\ccr$ is at most $\varepsilon\binom{t}{k}$.
Consequently, it follows from the degree conditions on~$\ccd$ and~$\ccn$ that all but at most $\sqrt{\varepsilon}\binom{t}{{k-2}}$ many $(k-2)$-sets~$K$ satisfy
\begin{equation*}
    \deg_{\ccr}(K)\geq \left(c-2d-\sqrt{\varepsilon}\right)\binom{t}{2}.
\end{equation*}
\end{proof}

We will find a suitable fractional $\ccc_{\ell}$-tiling in the reduced hypergraph~$\ccr$, where
the cherry~$\ccc_\ell$ is the $k$-uniform hypergraph with vertex set $[2k-2\ell]$ and edges $\{1,\dots,k\}$ and $\{k-2\ell+1,\dots,2k-2\ell\}$.

\begin{definition}
Let $\ccc$ and $\ccr$ be $k$-uniform hypergraphs, $\beta > 0$, and let $\Phi$ be a multiset of hypergraph homomorphisms from $\ccc$ to $\ccr$.
A function $h\colon \Phi \to \{a\beta \colon a \in \bbn_{>0}\}$ is called a $\beta$-$\homit(\ccc)$-tiling if the weight $w_h(v)$ of a vertex $v$ satisfies
\[
    w_h(v) = \sum_{u \in V(\ccc)}\sum_{\phi \in \Phi: v = \phi(u)} h(\phi) \le 1
\]
for all $v \in V(\ccr)$.
We call
\begin{equation*}
    w(h) = \sum_{v \in V(\ccr)}w_h(v) = \sum_{\phi \in \Phi} h(\phi) |V(\ccc)|
\end{equation*}
the weight of the tiling.
\end{definition}

The following building block allows us to easily define a tiling on a single edge.
\begin{fact}
\label{fact:building-block}
    Given an edge $e = \{v_1,\dots,v_k\}$, there exists a $\frac{1}{2(k-\ell-1)}$-$\homit(\ccc_\ell)$-tiling $h$ that is non-zero only on $e$, such that $w_h(v_i) = 1$ for $i \in [k-2]$ and ${w_h(v_{k-1}) = w_h(v_k) = \frac{k-2}{2(k-\ell-1)}}$.
    Note that we may scale the weight of $h$ by any~$q\in(0,1]$ and obtain a $\frac{q}{2(k-\ell-1)}$-$\homit(\ccc_\ell)$-tiling with ${w_h(v_i) = q}$ for $i \in [k-2]$ and $w_h(v_{k-1}) = w_h(v_k) = \frac{q(k-2)}{2(k-\ell-1)}$.
    Similarly, for any $q \in (0,1]$ there exists a $\frac{q}{2(k-\ell)}$-$\homit(\ccc_\ell)$-tiling with ${w_h(v_i) = q}$ for $i \in [k]$.
\end{fact}
\begin{proof}
For this consider the homomorphism that maps $\ccc_\ell$ to $e$ such that $v_1,\dots,v_{2\ell-2}, v_{k-1}$ and $v_k$ are the image of the intersection of the two edges of $\ccc_\ell$.
By cyclically shifting the image of the first $2\ell-2$ vertices of the intersection and appropriate scaling,  we obtain all homomorphisms for the required tiling.
We obtain the even weight distribution for the last part of the fact by cyclically shifting the whole image $k$ times.
\end{proof}
The following lemma is the main part of the proof of the Path-Tiling Lemma.
For this we introduce a fractional notion of extremality.
We say that a $k$-uniform hypergraph~$\ccr$ on~$t$ vertices is \emph{$\beta$-fractionally $(\ell,\xi)$-extremal} if there is a
function $b\colon V(\ccr) \to \{0\} \cup [\beta,1]$ with
\[
\sum_{v \in V(\ccr)} b(v) \geq \frac{2(k-\ell)-1}{2(k-\ell)}t
\qqand
    \sum_{e \in E(\ccr)} \prod_{v \in e} b(v) \leq \xi \binom{t}{k}.
\]
Note that the function $b$ can be viewed as a set of weighted vertices, which plays the r\^ole of the vertex set $B$ in the definition of extremality.

\begin{lemma}
\label{lemma:hom-tiling}
For all integers $k \geq 3$ and $1 \leq \ell < k/2$, there exist $C$ and $\gamma_0$ such that for all $\alpha > 0$ and $\gamma\in(0,\gamma_0)$, there exist $\beta>0$ and $\varepsilon>0$ such that the following holds for sufficiently large~$t$.
Let $\ccr$ be a $k$-uniform hypergraph on $t$ vertices that is not $\beta$-fractionally $(\ell,C\gamma)$-extremal and
\begin{equation}
    \label{eq:degcond}
    \deg(K)\geq\left(\frac{4(k-\ell)-1}{4(k-\ell)^2}-\gamma\right)\binom{t}{2}
\end{equation}
holds for all but at most $\varepsilon \binom{t}{k-2}$ sets $K \in \binom{V(\ccr)}{k-2}$.
Then there exists a $\beta$-$\hom(\ccc_\ell)$-tiling $h$ with weight at least $(1-\alpha)t$.
\end{lemma}

\begin{proof}
Clearly, it is sufficient to prove the lemma for small values of $\alpha$.
Consequently the quantification of the lemma allows us to fix the parameters and auxiliary constants~$C'$ and~$c$ to satisfy the following hierarchy of constants
\begin{equation}
    \label{eq:consthier}
    \frac{1}{k}, \frac{1}{\ell} \gg \frac{1}{C'} \gg \frac{1}{C} \gg \gamma_0 \ge \gamma \gg \alpha \gg c,\varepsilon,
\end{equation}
where ``$\gg x$'' denotes that $x$ is chosen sufficiently small with regard to all constants to its left.
Moreover, we fix $\beta$ inductively such that
\[
    1 = \beta_0 \gg \beta_1 \gg \dots \gg \beta_{\lfloor 1/c \rfloor} = \beta
    \qand 16\cdot k! \text{ divides } \tfrac{\beta_{i}}{\beta_{i+1}},
\]
and let $t$ be sufficiently large such that $c$, $\varepsilon$, $\beta\gg 1/t$.
Note that any $\beta_i$-$\hom(\ccc_\ell)$-tiling is also a $\beta$-$\hom(\ccc_\ell)$-tiling as $\beta_i$ is a multiple of $\beta$.
To prove the lemma, we show that given a $\beta_i$-$\hom(\ccc_\ell)$-tiling $h$ with weight $w(h) < (1-\alpha)t$, there exists a $\beta_{i+1}$-$\hom(\ccc_\ell)$-tiling $h'$ with weight $w(h') \ge w(h) + ct$.
We can begin with the trivial $1$-$\hom(\ccc_\ell)$-tiling with weight zero and hence, after at most $1/c$ steps, we obtain a $\beta$-$\hom(\ccc_\ell)$-tiling with weight at least $(1-\alpha)t$.

For the rest of the proof fix a $\beta_i$-$\hom(\ccc_\ell)$-tiling~$h$ with weight $w(h) < (1-\alpha)t$ and assume for a contradiction that there is no $\beta_{i+1}$-$\hom(\ccc_\ell)$-tiling with weight $w(h) + ct$.
It follows from the upper bound on the weight that there are at least $\alpha t/2$ vertices $v \in V(\ccr)$ with $w_h(v) < 1-\alpha/2$ and we may fix a subset $W$ of them of size $\alpha t/2$.

We view $\Phi$, the set of homomorphisms from $\ccc_\ell$ to $\ccr$, as a multiset, where we include~$\phi$ with multiplicity $\frac{h(\phi)}{\beta_i}$, so that we can assume $h\colon \Phi \to \{\beta_i\}$.
For our argument, we will need copies of~$\ccc_\ell$ to cover most of the hypergraph~$\ccr$ apart from~$W$.
So choose~$\Psi$, a multiset of functions from~$\ccc_\ell$ to $V(\ccr) \setminus W$, such that for all vertices~$v \in V(\ccr)$
\[
    \beta_i \cdot \sum_{i \in [2k-2\ell]} \left(|\{ \phi \in \Phi: v = \phi(i) \}| + |\{ \phi \in \Psi: v = \phi(i) \}|\right)
    \le
    1,
\]
and
\begin{equation}
    \label{eq:Phi}
    (1-2\alpha) \frac{t}{\beta_i v(\ccc_\ell)} \le |\Phi \cup \Psi| < (1-\alpha) \frac{t}{\beta_i v(\ccc_\ell)}.
\end{equation}
The right hand side of~\eqref{eq:Phi} holds for empty $\Psi$ and to see that $\Psi$ satisfying the inequality on the left can be chosen, note that one may simply take constant functions that map $\ccc_\ell$ to single vertices $v \in V(\ccr) \setminus \ccw$ for which $w_h(v) < 1 - \alpha$.
Let
\[
    \Xi = \Phi \cup \Psi
\]
and identify a function $\phi$ in $\Xi$ with the -- not necessarily distinct -- vertices $(v_1, \dots, v_{2k-2\ell})$ in its image, where $v_i = \phi(i)$ so that $\{v_1,\dots,v_k\}$ and $\{v_{k-2\ell+1},\dots,v_{2k-2\ell}\}$ form edges in~$\ccr$ if $\phi \in \Phi$.
We refer to the elements of $\Xi$ as \emph{cherries} $\ccc \in \Xi$.

Consider the $(k-2)$-sets in $W$ that satisfy the degree condition~\eqref{eq:degcond} of the lemma.
Since $\varepsilon \ll \alpha$, among those $(k-2)$-sets we find a collection $\ccw$ whose elements are pairwise disjoint and cover at least $|W|/2$ vertices. For later reference we note
\begin{equation}
    \label{eq:ccw}
    |\ccw|\geq\frac{|W|}{2(k-2)}
    >\frac{\alpha t}{4k}.
\end{equation}
For $K \in \ccw$ we consider the \emph{link graph} $L_K$ of $K$ in $\ccr$, which is the ($2$-uniform) graph containing all edges~$e$ such that $K \cup e \in E(\ccr)$.
At most $\frac{t}{\beta_i}\binom{v(\ccc_\ell)}{2} \leq \gamma \binom{t}{2}$ edges have both ends in the same~\mbox{$\ccc \in \Xi$} and at most $\alpha t^2/2 \le \gamma \binom{t}{2}$ edges contain a vertex from $W$, so let~$L'_K$ be the graph obtained from~$L_K$ by removing all these edges.
Combined with the degree condition~\eqref{eq:degcond} we have
\begin{equation}
    \label{eq:eL'K}
    e(L'_K) \ge \left(\frac{4(k-\ell)-1}{4{(k-\ell)}^2}-3\gamma\right) \binom{t}{2}
\end{equation}
for every such $(k-2)$-set~$K\in\ccw$.

We will find pairs $\ccc, \ccc' \in \Xi$ that allow us to locally improve the tiling $h$.
For this we only want to consider edges in the \emph{bipartite induced link graph $L_K(\ccc,\ccc')$}.
Formally the vertex classes of $L_K(\ccc,\ccc')$ are given by two disjoint copies of $[2k-2\ell]$.
In particular, $L_K(\ccc,\ccc')$ has $4k-4\ell$ vertices even when $\ccc$ and $\ccc'$ intersect or when $\ccc$ or $\ccc'$ are not given by injective functions from $\ccc_\ell$.
Moreover, two vertices $i$ and $j$ from different classes are adjacent in $L_K(\ccc,\ccc')$ if $\{v_i,v'_j\}$ is an edge in the link graph $L_K$, where $v_i$ is the image of $i\in V(\ccc_\ell)$ in $\ccc$ and $v'_j$ is the image of $j$ in $\ccc'$.
However, similarly as above we canonically identify the vertices of $L_K(\ccc,\ccc')$ with the vertices of $\ccc$ and $\ccc'$.

We show in the following that for most $K \in \ccw$ the bipartite link graph between most~$\ccc$ and $\ccc'$ has a very specific structure.
We call $(\ccc, \ccc') \in \Xi^2$ an \emph{extremal pair for $K$} if there exist \emph{special vertices} $u \in \ccc$ and $u' \in \ccc'$ such that $L_K(\ccc,\ccc')$ contains exactly all edges incident to these two vertices.
In particular, in such a case $L_K(\ccc,\ccc')$ has $4(k-\ell)-1$ edges.

\begin{claim}
\label{claim:tiling}
    Either there exists a $\beta_{i+1}$-$\hom(\ccc_\ell)$-tiling $h'$ with $w(h') > w(h) + ct$, or for every $\ccc \in \Xi$ there exists $u_\ccc \in \ccc$ such that the following holds.
    For all but at most $\gamma|\ccw|$ sets $K \in \ccw$ all but at most $C'\gamma |\Xi|^2$ pairs $(\ccc,\ccc') \in \Xi^2$ are extremal for~$K$ with special vertices~$u_{\ccc}$ and $u_{\ccc'}$.
\end{claim}
\begin{proof}
The proof of the claim consists of three steps.
First we show that if for a given $(k-2)$-tuple $K\in\ccw$ and some pair of cherries $\ccc,\ccc' \in \Xi$ the induced bipartite link graph~$L_K(\ccc, \ccc')$ contains a matching of size three or two vertices in $\ccc$ each neighbour to two distinct vertices in $\ccc'$, then there is a \emph{local improvement} of the tiling by a weight of at least $\beta_i/4$.
In a second step we shall bound the number of possible local improvements, as otherwise we could combine them to arrive at a desired tiling~$h'$ with a weight increased by~$ct$, which would conclude the proof.
With some foresight, we remark that every cherry~$\ccc \in \Xi$ may be used at most once for a local improvement.
In the last step we utilise this bound on the number of local improvements to show that ``typically''~$L_K(\ccc, \ccc')$ contains only $4(k - \ell) - 1$ edges and displays the structural conditions stated in the claim.

Let $K\in\ccw$ and $(\ccc$, $\ccc')\in\Xi^2$ be such that $w_h(v) \le 1 - \beta_i$ for all vertices $v \in K$.
For the first step we consider two cases.
Suppose that there is a matching with three edges in $L_K(\ccc,\ccc')$.
Recall that $L_K(\ccc,\ccc')$ is a bipartite graph with partition classes of size $2(k-\ell)$.
We set~$h'(\ccc)=(1-\frac{k-2}{6(k-\ell-1)})\beta_i$ if $\ccc \in \Phi$ and similarly $h'(\ccc')=(1-\frac{k-2}{6(k-\ell-1)})\beta_i$ if $\ccc' \in \Phi$.
If one or both of $\ccc, \ccc'$ are in $\Psi$, their vertices were not saturated by the $\beta_{i}$-$\hom(\ccc_\ell)$-tiling~$h$ and otherwise we have reduced their weight by~$\frac{k-2}{6(k-\ell-1)}\beta_i$.
So we can assign weight $\frac{k-2}{6(k-\ell-1)}\beta_i$ to the vertices of the matching edges and weight $\beta_i$ to the vertices of $K$.
This defines a valid $\beta_{i+1}$-$\hom(\ccc_\ell)$-tiling $h'$ by applying Fact~\ref{fact:building-block} (with $q=\frac 13\beta_i$) to the three edges in $\ccr$ corresponding to the matching edges in the link graph.
Note that the weights on the~$6$ vertices of the matching edges remain unchanged, and by considering the vertices on which the weight has changed, it is easy to see that
\[
    w(h') = w(h) + \left( k-2 - (4k -4\ell -6) \cdot\frac{k-2}{6(k-\ell-1)} \right) \beta_i
    \geq w(h) + \frac{1}{3}\beta_i\,,
\]
which yields a local improvement in this case.

For the next case suppose that there are two vertices in $\ccc$ each incident to two edges such that all four neighbours in $\ccc'$ are distinct.
Set $h'(\ccc)=(1-\frac{k-2}{4(k-\ell-1)})\beta_i$ if $\ccc \in \Phi$ and $h'(\ccc')=(1 - \frac{k-2}{8(k-\ell-1)})\beta_i$ if $\ccc' \in \Phi$.
On the vertices of the four edges described above we put weights $\frac{k-2}{8(k-\ell-1)}\beta_i$ and $\beta_i$ on the vertices of $K$.
Again, this defines a tiling $h'$ with
\begin{align}
    w(h') &= w(h) + \left( k-2 -(2k -2\ell -2) \cdot\frac{k-2}{4(k-\ell-1)} - (2k -2\ell -4) \cdot\frac{k-2}{8(k-\ell-1)} \right) \beta_i \nonumber\\
          &\geq w(h) + \frac{k-2}{4}\beta_i \geq w(h) + \frac{1}{4}\beta_i\,.\label{eq:step1case2}
\end{align}
This establishes a local improvement for this case and concludes the discussion of the first step.

For the second step suppose that there is a subset $\ccw'\subset\ccw$ of size at least $\gamma |\ccw|/2$, such that for each $K\in\ccw'$ we can define a local improvement for~$\gamma |\Xi|^2$ cherry pairs.
We apply these local improvements greedily, only using each cherry $\ccc\in \Xi$ at most once (over all $K\in\ccw'$), to increase the weight of the tiling.
This procedure may end, either when every $K\in\ccw'$ contains a saturated vertex, in which case we enlarge the total weight by at least
\[
    \frac{\alpha}{2}\cdot|\ccw'|
    \ge
    \frac{\alpha}{2}\cdot\frac{\gamma}{2}|\ccw|
    \overset{\eqref{eq:ccw}}{\ge}
    \frac{\alpha}{2}\cdot\frac{\gamma\alpha t}{8k},
\]
or when for at least one $K\in\ccw'$ for each of the $\gamma |\Xi|^2$ pairs of cherries at least one cherry was used for some local improvement already.
Any two cherries are contained in at most~$4|\Xi|$ ordered pairs so the latter case would imply that we applied $\gamma|\Xi|/4$ local improvements before.
In summary, we can aggregate local improvements leading to a $\beta_{i+1}$-$\hom(\ccc_\ell)$-tiling~$h''$ with weight at least
\[
    w(h'')\geq
    w(h)+
    \min\Big\{
        \frac{\alpha}{2}\cdot\frac{\gamma\alpha t}{8k}\,,\
        \frac{\beta_i}{4}\cdot\frac{\gamma}{4}|\Xi|
    \Big\}
    \overset{\eqref{eq:consthier},\eqref{eq:Phi}}{>}
    w(h)+ct,
\]
which would conclude the proof of Claim~\ref{claim:tiling}.

Consequently, for the third step we only need to consider those $K\in \ccw$ for which we can define a local improvement for less than $\gamma |\Xi|^2$ of its cherry pairs.
In particular, for those~$K$ most pairs induce no matching of size three in $L_K(\ccc,\ccc')$ and by K\"onig's theorem~\cite{Ko31} $L_K(\ccc,\ccc')$ spans at most $4(k-\ell)$ edges.
If it contains exactly $4(k-\ell)$ edges, the second local improvement considered in the first step would be possible, so indeed these pairs contain at most $4(k-\ell)-1$ edges.
On the other hand, in view of~\eqref{eq:Phi} the degree condition~\eqref{eq:eL'K} of $K \in \ccw$ translates to an average number of edges of at least~$4(k -\ell) - 1 - 4(3\gamma+4\alpha){(k-\ell)}^2$ in the link graphs.
So, as~$C'$ was chosen big enough, all but $(C'-1)\gamma |\Xi|^2$ cherry pairs~$\ccc$,~$\ccc'$ induce exactly $4(k - \ell) - 1$ edges in $L_K(\ccc,\ccc')$.
Since in addition these pairs allow no local improvement as considered in~\eqref{eq:step1case2}, there must be a vertex on each side that has a complete neighbourhood on the other side, so most pairs are indeed extremal.

It remains to show that typically the special vertex $u \in \ccc$ in an extremal pair $L_K(\ccc,\ccc')$ is independent of $K$ and $\ccc'$.
So assume for a moment that there are two vertices $u$ and~$v$ in $\ccc\in\Xi$ such that $u$ is a special vertex for an extremal pair~$L_K(\ccc,\ccc')$ and  $v$ is special for an extremal pair $L_{K'}(\ccc,\ccc'')$ for some (possibly non-distinct) $K$, $K'\in\ccw$ and $\ccc'$, $\ccc''\in\Xi$.
In this case we can define a local improvement by ``splitting'' the case with four edges above.
Indeed choose four edges incident with $u$ in $L_K(\ccc,\ccc')$ and four for~$v$ in~$L_{K'}(\ccc,\ccc'')$.
Assign weights $\frac 12 \beta_i$ to the vertices of $K$ and $K'$, and $\frac{k-2}{16(k-\ell-1)}\beta_i$ to the vertices of the eight chosen edges.
Set $h'(\ccc)=(1-\frac{k-2}{4(k-\ell-1)})\beta_i$ if $\ccc \in \Phi$ and reduce the weights on~$\ccc'$ and $\ccc''$ by $\frac{k-2}{16(k-\ell-1)}\beta_i$ if they are in $\Phi$ (or by $\frac{k-2}{8(k-\ell-1)}\beta_i$ in case $\ccc'=\ccc''$).
Similar calculations as in~\eqref{eq:step1case2} lead to a local improvement of $\beta_i/4$ involving the three cherries $\ccc$, $\ccc'$, and $\ccc''$.

For each cherry $\ccc$ fix $u_\ccc \in \ccc$ as the vertex that occurs most often as a special vertex over all extremal pairs $L_K(\ccc,\ccc')$.
Assume for a moment that at least for $\gamma |\ccw|/2$ many~$K \in \ccw$ for at least~$\gamma|\Xi|^2$ extremal pairs $\ccc$, $\ccc'$ the special vertex in $\ccc$ is not $u_\ccc$.
In particular for each such $K$ we find $\gamma|\Xi|/4$ such pairs none of which share a cherry.
By the choice of~$u_\ccc$ there exist $K'$, $\ccc''$ as above that allow us to define a local improvement as long as we have not applied more than $\gamma|\Xi|/12$ local improvements.
So we can aggregate the local improvements as in the second step.
Otherwise the chosen $u_\ccc$ satisfy the statement of the claim.
\end{proof}

We call $\ccc \in \Xi$ \emph{good} if it is contained in at least $\frac 12|\Xi|$ extremal pairs for at least $\frac 12 |\ccw|$ many $K \in \ccw$ and \emph{bad} otherwise.
As a $\beta_{i+1}$-$\hom(\ccc_\ell)$-tiling $h'$ with $w(h') > w(h) + ct$ would complete the proof of Lemma~\ref{lemma:hom-tiling}, Claim~\ref{claim:tiling} implies that there are at most $(C' + 1)\gamma |\ccw| |\Xi|^2$ triples $(K,\ccc, \ccc') \in \ccw\times\Xi^2$ such that $\ccc, \ccc'$ are not extremal for $K$.
So at most $5C'\gamma |\Xi|$ cherries are bad as we would have
\[
    5C'\gamma |\Xi| \cdot \frac 12 |\ccw| \cdot \frac 12 |\Xi|>(C' + 1)\gamma |\ccw| |\Xi|^2
\]
such triples otherwise.
Moreover, for every vertex $v\in V$ we denote by $\Xibad(v)$ the set of bad cherries $\ccc\in\Xi$ that contain it.

To complete the proof of Lemma~\ref{lemma:hom-tiling} we will show that we find a large matching~$M$ in $\ccr$ such that every vertex $v\in e\in M$ is contained in ``many'' good cherries.
For each good cherry~$\ccc \in \Xi$ there are a lot of choices for~$\ccc'$ and $K\in \ccw$ such that~$\ccc$ and~$\ccc'$ are an extremal pair for~$K$.
We will redistribute the weights to transfer weight from the non-special vertices of~$\ccc$ (and~$\ccc'$) to~$K$, which will reduce the weight on~$v$ (since we will ensure that $v$ is a non-special vertex).
Repeating this for every $v\in e$ will allow us to obtain a local improvement for the tiling by adding weight on $e$ and repeating this for sufficiently many hyperedges $e\in M$ leads to the desired global improvement.

We define the function $a \colon V(\ccr) \to [0,1]$ by $v \mapsto\beta_i \cdot \sum_{\ccc \in \Xi} \mathds{1}_{\{v\}}(u_\ccc)$, which assigns to a vertex the sum of weights of cherries that use it as a special vertex.
As any cherry contains $2(k-\ell)$ vertices, it is clear that $\sum_{v \in V(\ccr)} a(v) \leq \frac{t}{2(k-\ell)}$ and, therefore, we can utilise the $\beta$-fractional non-extremality of~$\ccr$ for $b(\cdot)=1-a(\cdot)$
and obtain
\[
    \sum_{e \in E(\ccr)} \prod_{v \in e} b(v) \ge C\gamma \binom{t}{k}.
\]
Since there are at most  $5C'\gamma |\Xi|$ bad cherries, they contribute at most
\begin{equation}
    \label{eq:badv}
    \beta_i\sum_{v\in V(\ccr)}\big|\Xibad(v)\big|
    \leq
    \beta_i v(\ccc_\ell) \cdot 5 C'\gamma |\Xi|
    \overset{\eqref{eq:Phi}}{\leq}
    5 C' \gamma t
\end{equation}
to the overall weight of the $\beta_{i}$-$\hom(\ccc_\ell)$-tiling $h$.
We shall only use good cherries to redistribute weights for the desired $\beta_{i+1}$-$\hom(\ccc_\ell)$-tiling, so we consider the function $b'\colon V(\ccr)\to[0,1]$ given by
\[
    b'(v)=\max\big\{0,\, b(v)-\beta_i\cdot|\Xibad(v)|\big\}
\]
and in view of~\eqref{eq:badv} and $C'\ll C$ (cf.~\eqref{eq:consthier}) we have
\[
    \sum_{e \in E(\ccr)} \prod_{v \in e} b'(v) \ge \frac{C}{2}\gamma \binom{t}{k}.
\]
We will use an averaging argument to obtain a matching $M$ in $E(\ccr)$ with a suitable lower bound on~$\sum_{e \in M} \prod_{v \in e} b'(v)$.
Consider all maximal matchings (of size $\lfloor t/k \rfloor$) in the complete $k$-uniform hypergraph $\cck_t$ on $t$ vertices.
Any given edge is contained in a ${\lfloor t/k \rfloor}/{\binom{t}{k}}$ fraction of those matchings, and so by averaging there is a matching~$M' \subset E(\cck_t)$ with
\[
    \sum_{e \in M' \cap E(\ccr)} \prod_{v \in e} b'(v) \ge \frac{C}{2}\gamma \cdot \left\lfloor\frac{t}{k}\right\rfloor \ge \frac{C}{3}\gamma \cdot \frac{t}{k}.
\]
Set $M = M' \cap E(\ccr)$.
Since~$b'(v)\in[0, 1]$ we have
\begin{equation}\label{eq:minM}
    \sum_{e \in M} k \cdot \min_{v\in e} \{b'(v)\}\ge \sum_{e \in M} k \prod_{v \in e} b'(v)\geq \frac{C}{3}\gamma t.
\end{equation}
In particular, we may assume that $\min_{v\in e}\{b'(v)\}>0$ for every $e\in M$, since this has no effect on inequality~\eqref{eq:minM}.
Moreover, from the definition of the function $b'(\cdot)$ it then follows that $\min_{v\in e}\{b'(v)\}\geq\beta_i$ for every $e\in M$.

For each vertex $v \in \bigcup M$, we consider good cherries that contain $v$ as a non-special vertex.
Assume that we have $K\in\ccw$ and an extremal pair $\ccc$, $\ccc'$ such that $v$ is a non-special vertex in $\ccc$.
Recall that $L_K(\ccc,\ccc')$ contains all edges incident to the two special vertices.
We define a \emph{local weight shift}:
If $\ccc \in \Psi$, we can increase the weight at the vertex $v$ by $\beta_i$, if $\ccc \in \Phi$ we will shift the weight as follows.
Assign weights $\frac{1}{2(k-\ell)-1} \cdot \frac{k-2}{4(k-\ell-1)}\beta_i$ to the vertices of all edges incident with exactly one of the special vertices, $\beta_i$ to the vertices of~$K$ and set $h'(\ccc)=(1-\frac{k-2}{4(k-\ell-1)})\beta_i$ and $h'(\ccc')=(1-\frac{k-2}{4(k-\ell-1)})\beta_i$ if $\ccc \in \Phi$.
By similar calculations as before, this defines a valid $\beta_{i+1}$-$\hom(\ccc_\ell)$-tiling $h'$ with $w(h')=w(h)$.
On the other hand, the weight of the vertex~$v$ is reduced by $\frac{k-2}{4(k-\ell)-2}\beta_i$, i.e.
\[
    w_{h'}(v)=w_{h}(v) - \frac{k-2}{4(k-\ell)-2}\beta_i\,.
\]
It follows from the definition of~$b'(v)$ that we have at least $b'(v)/\beta_i$ many good cherries that contain~$v$ as a non-special vertex and we shall apply at most $\min_{u\in e}\{b'(u)\}/\beta_i$ local weight shifts for a vertex $v\in e\in M$.

For every edge $e\in M$ we would like to apply these local weight shifts for every vertex $v\in e$, where we cycle through all $k$ vertices and apply one shift at a time.
In other words, we evenly reduce the weights on the vertices of $e$.
Note that we can apply these local weight shifts using $K, \ccc$, and $\ccc'$ unless we have saturated the vertices in $K$ or used one of the cherries before.
The procedure stops as soon as we reach a vertex for which no local weight shift is possible.

We first discuss the ideal case that this procedure does not stop, i.e.\ for every $e\in M$ and every $v\in e$ we applied $\min_{u\in e}\{b'(u)\}/\beta_i$ local weight shifts.
In this case, for every $e\in M$ we reduced the weight of all vertices $v\in e$ by at least
\[
    \frac{1}{\beta_i}\min_{u\in e}\{b'(u)\}\cdot \frac{k-2}{4(k-\ell)-2}\beta_i
    =
    \frac{k-2}{4(k-\ell)-2}\min_{u\in e}\{b'(u)\}.
\]
Consequently, we may appeal to Fact~\ref{fact:building-block} to increase the tiling on the edge $e$ by the same amount.
Repeating this for all $e \in M$, we obtain a $\beta_{i+1}$-$\hom(\ccc_\ell)$-tiling $h''$ satisfying
\[
    w(h'') \geq
    w(h) + \sum_{e\in M} k \cdot \frac{k-2}{4(k-\ell)-2}\min_{u\in e}\{b'(u)\}
    \overset{\eqref{eq:minM}}{\geq}
    w(h) + \frac{C\gamma t}{3} \cdot
    \frac{k-2}{4(k-\ell)-2}
    \overset{\eqref{eq:consthier}}{\geq}
    w(h) + ct,
\]
which would conclude the proof of Lemma~\ref{lemma:hom-tiling} in this case.

In the case that the procedure stops, there is some $v\in V(M)$ and a good cherry~$\ccc$ for~$v$ such that~$\ccc$ cannot be used for a local weight shift for~$v$.
This means, since~$\ccc$ is a good cherry, that either $\frac{1}{2}|\ccw|$ many~$K\in \ccw$ contain a saturated vertex or that at least $\frac{1}{2}|\Xi|$~cherries were used in local weight shifts before.
In the case that $\frac{1}{2}|\ccw|$ many~$K\in \ccw$ contain a saturated vertex, each of these vertices was used in at least $\frac{\alpha}{2\beta_i}$ local weight shifts, so in total we have applied
\[
    \frac{1}{2}|\ccw| \cdot \frac{\alpha}{2\beta_i}
    \overset{\eqref{eq:ccw}}{\geq}
    \frac{\alpha t}{8k} \cdot \frac{\alpha}{2\beta_i}
\]
local weight shifts.
If on the other hand all $\frac 12 |\Xi|$ possible cherries $\ccc'$ were used in local weight shifts before, then we have applied at least $\frac 14 |\Xi|$ local weight shifts.
As in the ideal case, using Fact~\ref{fact:building-block}, we conclude that we can increase the tiling on the edges in $M$ and obtain a $\beta_{i+1}$-$\hom(\ccc_\ell)$-tiling $h''$ with
\[
    w(h'')
    \geq
    w(h) + \Big(\min\Big\{
        \frac{\alpha^2 t}{16k\beta_i},\,
        \frac{|\Xi|}{4}
    \Big\}- k\Big)
        \cdot \frac{(k-2)\beta_i}{4(k-\ell)-2}
    \overset{\eqref{eq:consthier},\eqref{eq:Phi}}{\geq}
    w(h) + ct,
\]
which concludes the proof of Lemma~\ref{lemma:hom-tiling}.
\end{proof}

Next we want to transfer the $\beta$-$\hom(\ccc_\ell)$-tiling of~$\ccr$ into a path-tiling of~$\cch$.
For that purpose we will use the following lemma from~\cite{HaZh15}*{Lemma 2.7}.

\begin{lemma}
\label{lemma:partition}
Fix $k \geq 3$, $1 \leq \ell < k/2$ and $\varepsilon,\ d > 0$ such that $d > 2\varepsilon$. Let $m > \frac{k^2}{\varepsilon^2(d - \varepsilon)}$. Suppose $\ccv = (V_1, \ldots, V_k)$ is an $(\varepsilon, d )$-regular $k$-tuple with
$$
|V_1| = \cdots = |V_{2\ell}| = m \mbox{ and }|V_{2\ell +1 }| = \cdots = |V_{k}| = 2m.
$$
Then there are at most $\frac{2k}{(d - \varepsilon)\varepsilon}$ vertex disjoint $\ell$-paths that together cover all but at most $2k\varepsilon m$ vertices of $\ccv$.\qed
\end{lemma}

Finally, by using Lemma~\ref{lemma:partition} on the edges of the $\beta$-hom$(\ccc_\ell)$-tiling of~$\ccr$ given by Lemma~\ref{lemma:hom-tiling}, we obtain a path-tiling from~$\cch$ of the desired size.

\begin{lemma}[Path-Tiling Lemma]
\label{lemma:tiling}
For all integers $k\geq 3$ and $1\leq\ell<k/2$, there exist $C,\gamma_0>0$ such that for all $\alpha>0$, $\gamma\leq \gamma_0$ there exists an integer $s$ such that the following holds for all sufficiently large~$n$.
Let~$\cch$ be a $k$-uniform hypergraph on $n$ vertices and
\begin{equation*}
\delta_{k-2}(\cch)\geq\left(\frac{4(k-\ell)-1}{4(k-\ell)^2}-\gamma\right)\binom{n}{2}.
\end{equation*}
Then either there is a family of at most $s$ disjoint $\ell$-paths that cover all but at most $\alpha n$ vertices of $\cch$ or $\cch$ is $(\ell, C\gamma)$-extremal.
\end{lemma}

\begin{proof}

Let $k \geq 3$ and $1 \leq \ell < k/2$ be given.
Let~$C'$ and~$\gamma_0'$ be the constants given by Lemma~\ref{lemma:hom-tiling} for $k$ and $\ell$.
Set $C = 6C'$ and $\gamma_0 = \frac{\gamma_0'}{4}$, and let $\alpha > 0$ and $\gamma \le \gamma_0$.
Following the quantification of Lemma~\ref{lemma:hom-tiling} with~$\frac{\alpha}{2}$ and~$\gamma$ we obtain $\beta$ and $\varepsilon'$ and a sufficiently large~$t_0$.
Let $\varepsilon$ be sufficiently small.
Then the weak regularity lemma (Lemma~\ref{lemma:regularity_lemma}) for $\varepsilon_0 = \frac{\beta\varepsilon}{2}<\gamma^2$ and~$t_0$ yields $T_0$.
Let $s$ be a sufficiently large constant.
Let $\cch$ be a $k$-uniform hypergraph on $n$ vertices such that
\[
    \delta_{k-2}(\cch)\geq\left(\frac{4(k-\ell)-1}{4{(k-\ell)}^2}-\gamma\right)\binom{n}{2}
\]
and $n$ is sufficiently large.
By the weak regularity lemma there exists an $\varepsilon_0$-regular partition $V_0 \dcup \dots \dcup V_t$ of~$\cch$ with $|V_1|=\dots =|V_t|=m$, $|V_0| \le \varepsilon_0n$ and $t_0\leq t\leq T_0$ and the corresponding reduced hypergraph $\ccr=\ccr(\varepsilon_0, \gamma)$ on~$t$ vertices satisfies, by
Lemma~\ref{lemma:degree_clustergraph},
\[
    \deg_\ccr(K)\geq\left(\frac{4(k-\ell)-1}{4{(k-\ell)}^2}-4\gamma\right)\binom{t}{2}
\]
for all but at most $\sqrt{\varepsilon_0} \binom{t}{k-2} \le \varepsilon' \binom{t}{k-2}$ many $(k-2)$-sets $K\in \binom{[t]}{k-2}$.
We split the remainder of the proof in two cases, depending on whether $\ccr$ is $\beta$-fractionally $(\ell,4C'\gamma)$-extremal.

Suppose that $\ccr$ is not $\beta$-fractionally $(\ell,4C'\gamma)$-extremal, so in particular it is not $\beta$-fractionally $(\ell,C'\gamma)$-extremal.
Then Lemma~\ref{lemma:hom-tiling} implies that there exists a $\beta$-hom$(\ccc_\ell)$-tiling $h$ of $\ccr$ with weight $(1-\frac{\alpha}{2})t$.
Let $\Phi^+$ be the set of homomorphisms~$\phi$ from $\ccc_\ell$ to $\ccr$ with $h(\phi)>0$, which implies in fact $h(\phi)\geq \beta$.
We will use Lemma~\ref{lemma:partition} to obtain $\ell$-paths covering almost all vertices of~$\cch$ and for this we split the partition classes according to the tiling $h$: let ${\{R_1^\phi, \ldots, R_{2k - 2\ell}^\phi\}}_{\phi \in \Phi^+}$ be a family such that for all $\phi \neq \phi' \in \Phi^+$
\begin{itemize}
    \item $R_i^{\phi} \subset V_{\phi(i)}$ for all $i \in  [2k - 2\ell]$,
    \item $R_i^{\phi} \cap R_j^{\phi'} = \emptyset$ for all $i, j \in [2k - 2\ell]$,
    \item $|R_i^{\phi}| =2 \lfloor \frac{h(\phi)m}{2}\rfloor$ for all $i \in  [2k - 2\ell]$.
\end{itemize}
For each $\phi \in \Phi^+$ and all $i \in \{k - 2\ell + 1, \ldots, k\}$ let $S_i^\phi \cup U_i^\phi = R_i^\phi$ be a partition of $R_i^\phi$ into two classes of equal size.
Note that, since $(V_{\phi(1)}^\phi, \ldots, V_{\phi(k)}^\phi)$ and $(V_{\phi(k-2\ell+1)}^\phi, \ldots, V_{\phi(2k-2\ell)}^\phi)$ are $(\frac{\beta\varepsilon}{2},d)$-regular for some $d \ge \gamma$
\[
    (R_1^\phi, \ldots, R_{k - 2\ell}^\phi, S_{k - 2\ell + 1}^\phi, \ldots, S_{k}^\phi) \qand (U_{k - 2\ell + 1}^\phi, \ldots, U_{k}^\phi, R_{k + 1}^\phi, \ldots, R_{2k - 2\ell}^\phi)
\]
are $(\varepsilon,d)$-regular, for some $d \ge \gamma$, where we used that $h(\phi)\geq \beta$ for all $\phi\in\Phi^+$.
Then, with Lemma~\ref{lemma:partition} we obtain at most $\frac{2k}{(\gamma - \varepsilon)\varepsilon}$ many $\ell$-paths that cover all but $k\varepsilon|R_i^\phi|$ vertices of $R_1^\phi, \dots, R_{2k-2\ell}^\phi$.
Applying this to each homomorphism $\phi \in \Phi^+$ we obtain at most $s$ many $\ell$-paths.

We claim that the number of vertices in $V(\cch)$ that are not covered by these $\ell$-paths is less then $\alpha n$.
For this note that the uncovered vertices are the vertices from the partition class~$V_0$, the vertices that are not contained in any $R_i^\phi$ and those vertices in some $R_i^\phi$ that are not contained in any $\ell$-path.
At most $\frac{\alpha}{2}n$ vertices are not in any $R_i^\phi$ due to the weight of the $\beta$-hom$(C_\ell)$-tiling $h$ and we lose at most $\frac{2t}{\beta}$ vertices due to the rounding in the definition of $R_i^\phi$.
The $\ell$-paths cover all but a $(k\varepsilon)$-fraction of vertices in $\bigcup_{i,\phi} R_i^\phi$.
Consequently the total number of uncovered vertices is at most
\[
    \varepsilon_0 n + \frac{\alpha}{2}n + \frac{2t}{\beta} + k\varepsilon n < \alpha n.
\]

Now suppose that $\ccr$ is $\beta$-fractionally $(\ell,4C'\gamma)$-extremal.
This means by definition that there is a function $b\colon V(\ccr) \to \{0\} \cup [\beta,1]$ with
\[
\sum_{v \in V(\ccr)} b(v) \geq \frac{2(k-\ell)-1}{2(k-\ell)}t
\qqand
    \sum_{e \in E(\ccr)} \prod_{v \in e} b(v) \leq 4C'\gamma \binom{t}{k}.
\]
For each $i\in [t]$ we fix a subset $A_i\subseteq V_i$ with $|A_i|=\lfloor b(i)|V_i| \rfloor$ and define $B=\bigcup_{i\in [t]}A_i$.
Thus, we can bound the number of edges on $B$ by those that are in $(\varepsilon_0, d)$-regular tuples for some $d \ge \gamma$, the edges that are in $k$-tuples which are not dense or not regular and those that contain two or more vertices from the same $A_i$:
\begin{align*}
  e_\cch (B) & \leq   \sum_{e \in E(\ccr)} \prod_{v \in e}\left( b(v) \frac{n}{t}\right) + \binom{t}{k} \gamma{\left(\frac{n}{t}\right)}^k + \varepsilon_0 \binom{t}{k} {\left(\frac{n}{t}\right)}^k + t \binom{n/t}{2}\binom{n}{k-2} \\
                    & \leq  4C'\gamma\binom{n}{k} + \gamma\binom{n}{k} + \varepsilon_0 \binom{n}{k} + \frac{k(k - 1)}{2t}\binom{n}{k} \\
                    & \leq   5C'\gamma\binom{n}{k}.
\end{align*}
Note that
\[
    |B|  \geq  \left( \frac{2(k - \ell)-1}{2(k - \ell)}t\right)(1 - \varepsilon_0)\frac{n}{t} - t\geq \left( \frac{2(k - \ell) -1}{2(k - \ell)} - \varepsilon_0\right)n.
\]
Therefore, by adding at most $\varepsilon_0 n$ vertices from $V\setminus B$ to $B$ we obtain a set $B'$ with $|B'|= \left\lfloor \frac{2(k - \ell) -1}{2(k - \ell)}n \right\rfloor$ such that
\[
    e_\cch(B') \leq e_{\cch}(B) + \varepsilon_0 n\binom{n}{k-1}\leq  6C'\gamma\binom{n}{k} = C\gamma\binom{n}{k},
\]
from which we conclude that $\cch$ is $(\ell,C\gamma)$-extremal.
\end{proof}

\section{Proof of Theorem~\ref{theorem:hamilton}}
\label{sec:mainproof}
Below we give the proof of the main technical result, which details the outline from Section~\ref{sec:pf-outline} and is based on the lemmas from the last section.

\begin{proof}[Proof of Theorem~\ref{theorem:hamilton}]
Let $0<\xi<1$ and let $k\geq 4$ and $1\leq \ell<k/2$ be integers. Let~$C$ and~$\gamma_0$ be given by the Path-Tiling Lemma (Lemma~\ref{lemma:tiling}) for $k$ and $\ell$.
Let $\gamma<\gamma_0$ be a sufficiently small constant, in particular we may assume $C\gamma \ll \xi$.
From the Absorbing Lemma (Lemma~\ref{lemma:absorbing}) for $\eta=\zeta=\gamma$, $k$ and $\ell$ we obtain~$\varepsilon$.
Following the quantification of the Path-Tiling Lemma for $\alpha=\varepsilon/2$ and $5\gamma$ we obtain an integer $s$.
We will use the Reservoir Lemma (Lemma~\ref{lemma:reservoir}) with $\eta=\frac{4(k-\ell)-1}{4{(k-\ell)}^2}-3\gamma$, $\varepsilon'=\min\{\varepsilon/2,\gamma\}$, $k$, and~$m=s+1$.
Let $n\in (k-\ell)\bbn$ be sufficiently large and let~$\cch$ be a $k$-uniform hypergraph on $n$ vertices.

Suppose $\cch$ is not $(\ell,\xi)$-extremal and
\begin{equation*}
    \delta_{k-2}(\cch)
    \geq
    \left(\frac{4(k-\ell)-1}{4(k-\ell)^2}-\gamma\right)\binom{n}{2}.
\end{equation*}
Let $\cca$ be the absorbing path obtained with the Absorbing Lemma and let $X_0$ and $Y_0$ be the ends of $\cca$.
Then $|V(\cca)|\leq \gamma n$ and $\cca$ has the following absorption property:
for every subset $U\subset V\setminus V(\cca)$ with $|U|\leq \varepsilon n$ and $|U|\in (k-\ell)\bbn$ there exists an $\ell$-path $\ccq\subset \cch$ such that $V(\ccq) = V(\cca) \cup U$ and $\ccq$ has the ends $X_0$ and $Y_0$.

Let $V'=(V\setminus V(\cca))\cup \{X_0,Y_0\}$ and let $\cch'=\cch[V']$ be the subhypergraph of $\cch$ induced by $V'$.
Note that
\begin{equation*}
    \delta_{k-2}(\cch')
    \geq
    \left(\frac{4(k-\ell)-1}{4(k-\ell)^2}-3\gamma\right)\binom{n}{2}.
\end{equation*}
The Reservoir Lemma guarantees the existence of a set $R\subset V'$ with $|R|\leq \varepsilon'n\leq \gamma n$ such that for every $j\leq s+1$ every family ${(X_i,Y_i)}_{i\in [j]}$  of mutually disjoint pairs of sets of~$\ell$ vertices can be connected by paths that contain vertices of $\bigcup_{i\in[j]} (X_i\cup Y_i)\cup R$ only.

Let $V''=V\setminus (V(\cca)\cup R)$ and let $\cch''=\cch[V'']$ be the subhypergraph of $\cch$ induced by~$V''$.
Then
\begin{equation*}
    \delta_{k-2}(\cch'')
    \geq
    \left(\frac{4(k-\ell)-1}{4(k-\ell)^2}-5\gamma\right)\binom{n}{2}.
\end{equation*}

Now we apply the Path-Tiling Lemma to $\cch''$ and either we obtain a family of at most~$s$ disjoint $\ell$-paths that cover all but at most $\alpha |V''|\leq \alpha n$ vertices of $\cch''$, or $\cch''$ is $(\ell,5C \gamma)$-extremal.
Set $n''=|V''|$ and suppose for a contradiction that $\cch''$ is $(\ell, 5C \gamma)$-extremal.
Then there exists a set $B\subset V''$ such that $|B|=\big\lfloor \frac{2(k-\ell)-1}{2(k-\ell)}n''\big\rfloor$ and $e(B)\leq 5C \gamma {(n'')}^k$.
By adding at most $n-n''\leq 2\gamma n$ vertices from $V\setminus B$ to $B$, we obtain a vertex set $B'\subset V$ such that $|B'|=\big\lfloor \frac{2(k-\ell)-1}{2(k-\ell)}n\big\rfloor$ and
\[
    e(B')\leq  5C \gamma {(n'')}^k + 2\gamma n\binom{n-1}{k-1} \leq \xi n^k,
\]
a contradiction to the fact that $\cch$ is not $(\ell,\xi)$-extremal.
Therefore, we may assume that there exist disjoint $\ell$-paths $\ccp_1,\ldots,\ccp_j$ with $j\leq s$ that cover all but at most $\alpha |V''|\leq \alpha n$ vertices of $\cch''$.

For all $i\in[j]$, we denote the ends of $\ccp_i$ by $X_i$ and $Y_i$.
Let $Y_{j+1}=Y_0$.
By using the Reservoir Lemma to connect the family $(X_i,Y_{i+1})_{0\leq i\leq j}$, we connect the $\ell$-paths $\cca,\ccp_1,\ldots,\ccp_j$ to an $\ell$-cycle $\ccc\subset \cch$.

Let $U=V\setminus V(\ccc)$ be the set of vertices not contained in $\ccc$, i.e.\ the vertices that were leftover in the reservoir $R$ or uncovered by the path-tiling.
We have $|U|\leq (\varepsilon' + \alpha)n\leq \varepsilon n$.
Furthermore, since $\ccc$ is an $\ell$-cycle and $n\in(k-\ell)\bbn$, we have $|U|\in (k-\ell)\bbn$.
Therefore, we can utilise the absorbing property of $\cca$ to replace $\cca$ in $\ccc$ by a path $\ccq$ with the same ends as $\cca$, obtaining a Hamiltonian $\ell$-cycle of $\cch$.
\end{proof}

\section{Acknowledgements}
We thank both referees for their detailed work and helpful comments.


\begin{bibdiv}
\begin{biblist}

\bib{BMSCH3b}{article}{
    author={Bastos, Josefran {de} Oliveira},
    author={Mota, Guilherme Oliveira},
    author={Schacht, Mathias},
    author={Schnitzer, Jakob},
    author={Schulenburg, Fabian},
    title={Loose Hamiltonian cycles forced by large $(k-2)$-degree -- sharp version},
    note={In preparation},
}

\bib{BuHaSc13}{article}{
   author={Bu{\ss}, Enno},
   author={H{\`a}n, Hi\d{\^e}p},
   author={Schacht, Mathias},
   title={Minimum vertex degree conditions for loose Hamilton cycles in
   3-uniform hypergraphs},
   journal={J. Combin. Theory Ser. B},
   volume={103},
   date={2013},
   number={6},
   pages={658--678},
   issn={0095-8956},
   review={\MR{3127586}},
   doi={10.1016/j.jctb.2013.07.004},
}

\bib{Ch91}{article}{
   author={Chung, Fan R. K.},
   title={Regularity lemmas for hypergraphs and quasi-randomness},
   journal={Random Structures Algorithms},
   volume={2},
   date={1991},
   number={2},
   pages={241--252},
   issn={1042-9832},
   review={\MR{1099803}},
   doi={10.1002/rsa.3240020208},
}

\bib{FR92}{article}{
   author={Frankl, P.},
   author={R{\"o}dl, V.},
   title={The uniformity lemma for hypergraphs},
   journal={Graphs Combin.},
   volume={8},
   date={1992},
   number={4},
   pages={309--312},
   issn={0911-0119},
   review={\MR{1204114}},
   doi={10.1007/BF02351586},
}

\bib{HaSc10}{article}{
   author={H{\`a}n, Hi\d{\^e}p},
   author={Schacht, Mathias},
   title={Dirac-type results for loose Hamilton cycles in uniform
   hypergraphs},
   journal={J. Combin. Theory Ser. B},
   volume={100},
   date={2010},
   number={3},
   pages={332--346},
   issn={0095-8956},
   review={\MR{2595675}},
   doi={10.1016/j.jctb.2009.10.002},
}

\bib{HaZh15}{article}{
   author={Han, Jie},
   author={Zhao, Yi},
   title={Minimum codegree threshold for Hamilton $\ell$-cycles in
   $k$-uniform hypergraphs},
   journal={J. Combin. Theory Ser. A},
   volume={132},
   date={2015},
   pages={194--223},
   issn={0097-3165},
   review={\MR{3311344}},
   doi={10.1016/j.jcta.2015.01.004},
}

\bib{HaZh15b}{article}{
   author={Han, Jie},
   author={Zhao, Yi},
   title={Minimum vertex degree threshold for loose Hamilton cycles in
   3-uniform hypergraphs},
   journal={J. Combin. Theory Ser. B},
   volume={114},
   date={2015},
   pages={70--96},
   issn={0095-8956},
   review={\MR{3354291}},
   doi={10.1016/j.jctb.2015.03.007},
}

\bib{KaKi99}{article}{
   author={Katona, Gyula Y.},
   author={Kierstead, H. A.},
   title={Hamiltonian chains in hypergraphs},
   journal={J. Graph Theory},
   volume={30},
   date={1999},
   number={3},
   pages={205--212},
   issn={0364-9024},
   review={\MR{1671170}},
}

\bib{KeKuMyOs11}{article}{
   author={Keevash, Peter},
   author={K{\"u}hn, Daniela},
   author={Mycroft, Richard},
   author={Osthus, Deryk},
   title={Loose Hamilton cycles in hypergraphs},
   journal={Discrete Math.},
   volume={311},
   date={2011},
   number={7},
   pages={544--559},
   issn={0012-365X},
   review={\MR{2765622}},
   doi={10.1016/j.disc.2010.11.013},
}

\bib{Ko31}{article}{
      author={{K\"onig}, Denes},
       title={Graphok \'es matrixok},
    language={Hungarian (German summary)},
        date={1931},
        ISSN={0302-7317},
     journal={{Mat. Fiz. Lapok}},
      volume={38},
       pages={116--119},
}

\bib{KuMyOs10}{article}{
   author={K{\"u}hn, Daniela},
   author={Mycroft, Richard},
   author={Osthus, Deryk},
   title={Hamilton $\ell$-cycles in uniform hypergraphs},
   journal={J. Combin. Theory Ser. A},
   volume={117},
   date={2010},
   number={7},
   pages={910--927},
   issn={0097-3165},
   review={\MR{2652102}},
   doi={10.1016/j.jcta.2010.02.010},
}

\bib{KuOs06}{article}{
   author={K{\"u}hn, Daniela},
   author={Osthus, Deryk},
   title={Loose Hamilton cycles in 3-uniform hypergraphs of high minimum
   degree},
   journal={J. Combin. Theory Ser. B},
   volume={96},
   date={2006},
   number={6},
   pages={767--821},
   issn={0095-8956},
   review={\MR{2274077}},
   doi={10.1016/j.jctb.2006.02.004},
}

\bib{RRsurv}{article}{
   author={R{\"o}dl, Vojtech},
   author={Ruci{\'n}ski, Andrzej},
   title={Dirac-type questions for hypergraphs---a survey (or more problems
   for Endre to solve)},
   conference={
      title={An irregular mind},
   },
   book={
      series={Bolyai Soc. Math. Stud.},
      volume={21},
      publisher={J\'anos Bolyai Math. Soc., Budapest},
   },
   date={2010},
   pages={561--590},
   review={\MR{2815614}},
   doi={10.1007/978-3-642-14444-8\_16},
}

\bib{RoRuSz06}{article}{
   author={R{\"o}dl, Vojt{\v{e}}ch},
   author={Ruci{\'n}ski, Andrzej},
   author={Szemer{\'e}di, Endre},
   title={A Dirac-type theorem for 3-uniform hypergraphs},
   journal={Combin. Probab. Comput.},
   volume={15},
   date={2006},
   number={1-2},
   pages={229--251},
   issn={0963-5483},
   review={\MR{2195584}},
   doi={10.1017/S0963548305007042},
}

\bib{RoRuSz08}{article}{
   author={R{\"o}dl, Vojt{\v{e}}ch},
   author={Ruci{\'n}ski, Andrzej},
   author={Szemer{\'e}di, Endre},
   title={An approximate Dirac-type theorem for $k$-uniform hypergraphs},
   journal={Combinatorica},
   volume={28},
   date={2008},
   number={2},
   pages={229--260},
   issn={0209-9683},
   review={\MR{2399020}},
   doi={10.1007/s00493-008-2295-z},
}

\bib{St90}{thesis}{
   author={Steger, Angelika},
   title={Die Kleitman-Rothschild Methode},
   type={PhD thesis},
   organization={Forschungsinstitut f\"ur Diskrete Mathematik, Rheinische Friedrichs-Wilhelms-Universit\"at Bonn},
   date={1990},
}


\bib{Sz75}{article}{
   author={Szemer{\'e}di, Endre},
   title={Regular partitions of graphs},
   language={English, with French summary},
   conference={
      title={Probl\`emes combinatoires et th\'eorie des graphes},
      address={Colloq. Internat. CNRS, Univ. Orsay, Orsay},
      date={1976},
   },
   book={
      series={Colloq. Internat. CNRS},
      volume={260},
      publisher={CNRS, Paris},
   },
   date={1978},
   pages={399--401},
   review={\MR{540024}},
}

\bib{MR3526407}{article}{
   author={Zhao, Yi},
   title={Recent advances on Dirac-type problems for hypergraphs},
   conference={
      title={Recent trends in combinatorics},
   },
   book={
      series={IMA Vol. Math. Appl.},
      volume={159},
      publisher={Springer},
   },
   date={2016},
   pages={145--165},
   review={\MR{3526407}},
   doi={10.1007/978-3-319-24298-9\_6},
}

\end{biblist}
\end{bibdiv}
\end{document}